\documentclass[10pt]{article}
\usepackage[utf8]{inputenc}
\usepackage{tikz-cd}
\usepackage{tikz}
\usepackage{graphicx}
\usepackage{amsfonts}
\usepackage{url}
\usepackage{amsmath,amsthm,amssymb,enumitem}
\usepackage{mathtools}
\usepackage{geometry}
\usepackage{mathrsfs} 
\mathtoolsset{showonlyrefs}
\usepackage{stmaryrd}

\usepackage{xcolor,palatino}
\definecolor{ultramarine}{RGB}{0,32,96}
\colorlet{mygreen}{green!20!gray}
\colorlet{myultramarine}{ultramarine!20!gray}

\usepackage[colorlinks=true,citecolor=mygreen, linkcolor=mygreen, urlcolor=myultramarine,breaklinks, naturalnames, pagebackref]{hyperref}

\usepackage{float}
\usepackage{comment}
\makeatletter
\newsavebox\myboxA
\newsavebox\myboxB
\newlength\mylenA

\DeclareMathOperator*{\dprime}{^{\prime \prime}}

\numberwithin{equation}{section}
\numberwithin{figure}{section}
\allowdisplaybreaks[4]

\newcommand{\upperset}[2]{%
\underset{%
        \text{\raisebox{0.5ex}{\smash{\fontsize{5}{5}$#1$}}}
              }{#2}%
                    }

\newcommand{\uppersetF}[2]{%
\underset{%
        \text{\raisebox{0.8ex}{\smash{\fontsize{5}{5}$#1$}}}
              }{#2}%
                    }
\newcommand{\oversetF}[2]{%
\overset{%
        \text{\raisebox{0.2ex}{\smash{\fontsize{5}{5}$#1$}}}
              }{#2}%
                    }

\DeclareFontFamily{U}{BOONDOX-calo}{\skewchar\font=45 }
\DeclareFontShape{U}{BOONDOX-calo}{m}{n}{
  <-> s*[1.05] BOONDOX-r-calo}{}
\DeclareFontShape{U}{BOONDOX-calo}{b}{n}{
  <-> s*[1.05] BOONDOX-b-calo}{}
\DeclareMathAlphabet{\mathcalboondox}{U}{BOONDOX-calo}{m}{n}
\SetMathAlphabet{\mathcalboondox}{bold}{U}{BOONDOX-calo}{b}{n}
\DeclareMathAlphabet{\mathbcalboondox}{U}{BOONDOX-calo}{b}{n}

\title{Homotopy Transfer Theorem and minimal models for pre-Calabi-Yau categories}
\author{Marion Boucrot}
\date{}

\newcommand{\MA}{\mathcal{A}}
\newcommand{\MB}{\mathcal{B}}
\newcommand{\MO}{\mathcal{O}}

\newcommand{\ZZ}{{\mathbb{Z}}}
\newcommand{\NN}{{\mathbb{N}}}
\newcommand{\kk}{\Bbbk}
\newcommand{\Hom}{\operatorname{Hom}}
\newcommand{\shom}{\operatorname{hom}}
\newcommand{\Homgr}{\mathcal{H}om}

\newcommand{\llg}{\operatorname{lg}}
\newcommand{\nec}{\operatorname{nec}}

\newcommand{\llt}{\operatorname{lt}}
\newcommand{\rrt}{\operatorname{rt}}
\newcommand{\Multi}{\operatorname{Multi}}

\newcommand{\id}{\operatorname{id}}
\newcommand{\multinec}{\operatorname{multinec}}
\newcommand{\pre}{\operatorname{pre}}
\makeatletter
\newcommand*{\relrelbarsep}{.386ex}
\newcommand*{\relrelbar}{%
  \mathrel{%
    \mathpalette\@relrelbar\relrelbarsep
  }%
}
\newcommand*{\@relrelbar}[2]{%
  \raise#2\hbox to 0pt{$\m@th#1\relbar$\hss}%
  \lower#2\hbox{$\m@th#1\relbar$}%
}
\providecommand*{\rightrightarrowsfill@}{%
  \arrowfill@\relrelbar\relrelbar\rightrightarrows
}
\providecommand*{\leftleftarrowsfill@}{%
  \arrowfill@\leftleftarrows\relrelbar\relrelbar
}
\providecommand*{\xrightrightarrows}[2][]{%
  \ext@arrow 0359\rightrightarrowsfill@{#1}{#2}%
}
\providecommand*{\xleftleftarrows}[2][]{%
  \ext@arrow 3095\leftleftarrowsfill@{#1}{#2}%
}
\newcommand*\xoverline[2][0.75]{%
    \sbox{\myboxA}{$\m@th#2$}%
    \setbox\myboxB\null
    \ht\myboxB=\ht\myboxA%
    \dp\myboxB=\dp\myboxA%
    \wd\myboxB=#1\wd\myboxA
    \sbox\myboxB{$\m@th\overline{\copy\myboxB}$}
    \setlength\mylenA{\the\wd\myboxA}
    \addtolength\mylenA{-\the\wd\myboxB}%
    \ifdim\wd\myboxB<\wd\myboxA%
       \rlap{\hskip 0.5\mylenA\usebox\myboxB}{\usebox\myboxA}%
    \else
        \hskip -0.5\mylenA\rlap{\usebox\myboxA}{\hskip 0.5\mylenA\usebox\myboxB}%
    \fi}
\makeatother

\newcommand*{\doubarl}[1]{\xoverline{\xoverline{#1}}} 
\newcommand*{\doubar}[1]{\bar{\bar{#1}}} 

\newtheorem{definition}{Definition}[section]
\newtheorem{definition-proposition}[definition]{Definition-Proposition}
\newtheorem{remark}[definition]{Remark}
\newtheorem{theorem}[definition]{Theorem}
\newtheorem{proposition}[definition]{Proposition}

\newtheorem{corollary}[definition]{Corollary}
\newtheorem{lemma}[definition]{Lemma}

\begin{document}
\maketitle
\hrulefill
\begin{abstract} 
In this article, we extend results of J. Leray and B. Vallette on homotopical properties of pre-Calabi-Yau algebras to the case of pre-Calabi-Yau categories. We give direct proofs of the results adapting techniques used by D. Petersen for the case of coalgebras, instead of using properadic calculus. More precisely, we prove that given quasi-isomorphic dg quivers and a pre-Calabi-Yau structure on one of them there exists a pre-Calabi-Yau structure on the other as well as a pre-Calabi-Yau morphism between the two.
Moreover, we show that a pre-Calabi-Yau morphism whose first component is an isomorphism of dg quivers is an isomorphism in the category of pre-Calabi-Yau categories. We incidentally show that any pre-Calabi-Yau category has a minimal model.
Finally, we prove that any
quasi-isomorphism of pre-Calabi-Yau categories admits a quasi-inverse.
\end{abstract}

\textbf{Mathematics subject classification 2020:} 16E45, 18G70, 14A22

\textbf{Keywords:} $A_{\infty}$-algebras, pre-Calabi-Yau algebras, homotopy transfer, minimal models

\hrulefill

\tableofcontents

\section{Introduction}

Pre-Calabi-Yau structures were introduced by M. Kontsevich and Y. Vlassopoulos in the last decade.
In the finite dimensional case, these structures have also appeared in \cite{tz}
under the name of $V_{\infty}$-algebras and P. Seidel called them $A_{\infty}$-algebras with boundary in \cite{seidel}.
They also appear under the name of weak Calabi-Yau structures in \cite{kontsevich}. These references show that pre-Calabi-Yau structures
play an important role in homological algebra, symplectic geometry, string topology, noncommutative geometry and even in Topological Quantum Field Theory. 

The notion of pre-Calabi-Yau algebras is strongly related to the notion $A_{\infty}$-algebras. Actually, for $d\in\ZZ$, a $d$-pre-Calabi-Yau structure on a finite dimensional graded vector space $A$ is equivalent to a cyclic $A_{\infty}$-structure on $A\oplus A^*[d-1]$ such that $A$ is an $A_{\infty}$-subalgebra.

In 1980, T. Kadeishvili proved in \cite{kadeishvili} the following theorem:
\begin{theorem}
\label{thm:HTT-kadeishvili}
    (Homotopy Transfer Theorem) For any dg algebra $C$ such that $H_i(C)$ is free for $i\geq 0$, there exists an $A_{\infty}$-structure on $H(C)$ with vanishing differential and an $A_{\infty}$-morphism $f:H(C)\rightarrow C$ such that $f^1$ is a quasi-isomorphism.
\end{theorem}

In 2020, D. Petersen published in \cite{petersen} a direct proof of the following theorem, which is a reformulation of the Homotopy Transfer Theorem:
\begin{theorem}
    \label{thm:HTT-petersen}
    Let $R$ be a ring, $(V,d_V)$ and $(W,d_W)$ dg $R$-modules and $f : V\rightarrow W$ a chain map.
    Let $\nu$ be a square-zero coderivation of $C(W)$ of degree $-1$ whose arity $1$ term is $d_W$.
    Assume that f induces a quasi-isomorphism $\Hom_R(C(V),V)\rightarrow \Hom_R(C(W),W)$.
    Then, there exists noncanonically a square-zero coderivation $\mu$ of $C(V)$ of degree $-1$ whose arity $1$ term is $d_V$ and a morphism $F : C(V)\rightarrow C(W)$ of coalgebras whose arity $1$ term is f and which is a chain map with respect to $\mu$ and $\nu$.
\end{theorem}

The converse statement also holds, namely, if $V$ is endowed with a square-zero coderivation then we can construct one on $W$. 
The Homotopy Transfer Theorem has been proved in several different ways and we refer the reader to the introduction of \cite{petersen} and the references therein for a more complete account about this result.

In \cite{lv}, J. Leray and B. Vallette show that the notion of curved homotopy double Poisson gebras is equivalent to the notion of curved pre-Calabi-Yau algebras.
Using their paper \cite{hlv} with E. Hoffbeck on properadic homotopical calculus, they obtain homotopical properties of pre-Calabi-Yau algebras. This includes a version of the Homotopy Transfer Theorem, namely, they show that a pre-Calabi-Yau structure on a locally finite dimensional dg vector space gives rise to a pre-Calabi-Yau structure on any of its contraction together with extensions of the chain maps into $\infty$-morphisms. 
Moreover, the previous authors show that an $\infty$-morphism whose first component is an isomorphism is invertible and deduce that $\infty$-quasi-isomorphisms are quasi-invertible. 

The aim of this article is to give direct proofs of the previous statements, extending them to the case of pre-Calabi-Yau categories (see Theorems \ref{thm:2-pcy}, \ref{thm:3-pcy} and \ref{thm:q-iso-pCY}), 
by adapting the techniques of \cite{petersen}.

Let us briefly present the contents of this paper.
Section \ref{not-conv} is devoted to present the conventions we follow and notations we will use throughout this paper.
In section \ref{A-inf-case} we recall the notions of $A_{\infty}$-categories and $A_{\infty}$-morphisms. We also recall well-known results including the ones appearing in \cite{petersen}, giving proofs entirely in terms of the structure of the $A_{\infty}$-category, instead of the associated bar constructions, as it is the case in the mentioned article. 
This is done in order to allow the comparison with the case of pre-Calabi-Yau categories.
We do not claim any originality in this regards.
Section \ref{pCY-case} is the core of the article. After recalling the definition of $d$-pre-Calabi-Yau categories and morphisms given in \cite{ktv}, we prove that a $d$-pre-Calabi-Yau morphism whose first component is an isomorphism is an isomorphism in the category of $d$-pre-Calabi-Yau categories (see Lemma \ref{lemma:iso-pCY}).
Then, we prove Theorems \ref{thm:2-pcy} and \ref{thm:3-pcy} by generalizing the proofs of the case of $A_{\infty}$-categories. Those results state that given dg quivers and a quasi-isomorphism between them, if one carries a $d$-pre-Calabi-Yau structure
then we can construct such a structure on the other, as well as a $d$-pre-Calabi-Yau morphism between the two. We deduce that there exists minimal models for $d$-pre-Calabi-Yau categories.
Finally, using Lemma \ref{lemma:iso-pCY} and Theorems \ref{thm:2-pcy} and \ref{thm:3-pcy}, we show that a quasi-isomorphism of $d$-pre-Calabi-Yau categories admits a quasi-inverse (see Theorem \ref{thm:q-iso-pCY}).

\textbf{Acknowledgements.} This work is part of a PhD thesis supervised by Estanislao Herscovich and Hossein Abbaspour. The author thanks them for the useful advice and comments about this paper. She also thanks J. Leray and B. Vallette for explaining their results, their availability to discuss and all the advice concerning the present paper. 

This work was supported by the French National Research Agency in the framework of the ``France 2030" program (ANR-15-IDEX-0002) and the LabEx PERSYVAL-Lab (ANR-11-LABX-0025-01).

\section{Notations and conventions} \label{not-conv}
In what follows $\kk$ will be a field of characteristic different from $2$ and $3$ and to simplify we will denote $\otimes$ for $\otimes_{\kk}$. We will denote by $\NN = \{0,1,2,\dots \}$ the set of natural numbers and we define $\NN^*=\NN\setminus\{0\}$.
For $i,j\in\NN$, we define the interval of integers $\llbracket i,j\rrbracket=\{n\in\NN | i\leq n\leq j\}$. 

Recall that if we have a (cohomologically) graded vector space $V=\oplus_{i\in\ZZ}V^i$, we define for $n\in\ZZ$ the graded vector space $V[n]$ given by $V[n]^i=V^{n+i}$ for $i\in\ZZ$ and the map $s_{V,n} : V\rightarrow V[n]$ whose underlying set theoretic map is the identity. 
Moreover, if $f:V\rightarrow W$ is a morphism of graded vector spaces, we define the map $f[n] : V[n]\rightarrow W[n]$ sending an element $s_{V,n}(v)$ to $s_{W,n}(f(v))$ for all $v\in V$.
We will denote $s_{V,n}$ simply by $s_n$ when there is no possible confusion, and $s_1$ just by $s$.

We now recall the Koszul sign rules, that are the ones we use to determine the signs appearing in this paper. If $V,W$ are graded vector spaces, we have a map
$\tau_{V,W} : V\otimes W\rightarrow W\otimes V$ defined as
\begin{equation}
    \tau_{V,W}(v\otimes w)=(-1)^{|w||v|}w\otimes v
\end{equation}
where $v\in V$ is a homogeneous element of degree $|v|$ and $w\in W$ is a homogeneous element of degree $|w|$. 
More generally, given graded vector spaces $V_1,\dots,V_n$ and $\sigma\in\mathcalboondox{S}_n$, 
we have a map
\[
\tau^{\sigma}_{V_1,\dots,V_n} : V_1\otimes\dots\otimes V_n \rightarrow V_{\sigma^{-1}(1)}\otimes\dots\otimes V_{\sigma^{-1}(n)}
\]
defined as
\begin{equation}
    \tau^{\sigma}_{V_1,\dots,V_n}(v_1\otimes\dots\otimes v_n)=(-1)^{\epsilon}(v_{\sigma^{-1}(1)}\otimes\dots\otimes v_{\sigma^{-1}(n)})
\end{equation}
with
\[
\epsilon=\sum\limits_{\substack{i>j\\ \sigma^{-1}(i)<\sigma^{-1}(j)}} |v_{\sigma^{-1}(i)}||v_{\sigma^{-1}(j)}|
\]
where $v_i\in V_i$ is a homogeneous element of degree $|v_i|$ for $i\in\llbracket 1,n \rrbracket$.

Throughout this paper, when we consider an element $v$ of degree $|v|$ in a graded vector space $V$, we mean a homogeneous element $v$ of $V$.
Also, we will denote by $\id$ the identity map of every space of morphisms, without specifying it.
All the products in this paper will be products in the category of graded vector spaces. Given graded vector spaces $(V_i)_{i\in I}$, we thus have \[
\prod\limits_{i\in I}V_i=\bigoplus\limits_{n\in\ZZ}\prod\limits_{i\in I}V_{i}^{n}\]
where the second product is the usual product of vector spaces.

Given graded vector spaces $V,W$ we will denote by $\Hom_{\kk}(V,W)$ the vector space of $\kk$-linear maps $f : V \rightarrow W$ and by $\shom_{\kk}^d(V,W)$ the vector space of homogeneous $\kk$-linear maps $f : V \rightarrow W$ of degree d, \textit{i.e.} $f(v)\in W^{n+d}$ for all $v\in V^n$.
We assemble them in the graded vector space $\Homgr_{\kk}(V,W) =\bigoplus_{d\in \ZZ} \shom_{\kk}^d(V,W)\subseteq \Hom_{\kk}(V,W)$. We define the \textbf{\textcolor{ultramarine}{graded dual}} of a graded vector space $V=\bigoplus_{n\in\ZZ}V^n$ as the graded vector space $V^{*}=\Homgr_{\kk}(V,\kk)$.
Moreover, given graded vector spaces $V$, $V'$, $W$, $W'$ and homogeneous elements $f\in \Homgr_{\kk}(V,V')$ and $g\in\Homgr_{\kk}(W,W')$, we have that 
\begin{equation}
(f\otimes g)(v\otimes w)=(-1)^{|g||v|}f(v)\otimes g(w)
\end{equation}
for homogeneous elements $v\in V$ and $w\in W$.
Recall that a graded quiver $\MA$ consists of a set of objects $\MO$ together with graded vector spaces ${}_y\MA_x$ for every $x,y\in\MO$.
A dg quiver $\MA$ is a graded quiver such that ${}_y\MA_x$ is a dg vector space for every $x,y\in\MO$.
Given a quiver $\MA$, its \textbf{\textcolor{ultramarine}{enveloping graded quiver}} is the graded quiver $\MA^{e}=\MA^{op}\otimes \MA$ whose set of objects is $\MO\times \MO$ and whose space of morphisms from an object $(x,y)$ to an object $(x',y')$ is 
defined as the graded vector space ${}_{(x',y')}(\MA^{op}\otimes \MA)_{(x,y)}={}_{x}\MA_{x'}\otimes {}_{y'}\MA_{y}$.
Given graded quivers $\MA$ and $\MB$ with respective sets of objects $\MO_{\MA}$ and $\MO_{\MB}$ , a \textbf{\textcolor{ultramarine}{morphism of graded quivers}} $(\Phi_0,\Phi) : \MA \rightarrow \MB$ is the data of a map $\Phi_0 : \MO_{\MA}\rightarrow \MO_{\MB}$ between the sets of objects together with a collection $\Phi=(\Phi^{x,y})_{x,y\in\MO_{\MA}}$ of morphisms of graded vector spaces $\Phi^{x,y} : {}_x\MA_y\rightarrow {}_{\Phi_0(x)}\MB_{\Phi_0(y)}$ for every $x,y\in\MO_{\MA}$.
A \textbf{\textcolor{ultramarine}{dg quiver}} is a graded quiver $\MA$ with set of objects $\MO$ together with a collection $d_{\MA}=(d_{\MA}^{x,y})_{x,y\in\MO}$, where $({}_{x}\MA_y,d_{\MA}^{x,y})$ is a dg vector space for each $x,y\in\MO$, called the \textbf{\textcolor{ultramarine}{differential of $\MA$}}. To simplify, we will often omit the objects when writing the differential.
Given dg quivers $(\MA,d_{\MA})$ and $(\MB,d_{\MB})$ with respective sets of objects $\MO_{\MA}$ and $\MO_{\MB}$ , a \textbf{\textcolor{ultramarine}{morphism of dg quivers}} is a morphism of graded quivers $(\Phi_0,\Phi) : \MA \rightarrow \MB$ such that $\Phi^{x,y}$ is a chain map for every $x,y\in\MO_{\MA}$. 

Recall that given a dg quiver $(\MA,d_{\MA})$ with set of objects $\MO$ its \textbf{\textcolor{ultramarine}{cohomology}} is the graded quiver $H(\MA)$ whose objects are those of $\MA$ and whose spaces of morphisms are defined by 
\[
{}_{x}H(\MA)_{y}=\bigoplus\limits_{n\in\ZZ}{}_{x}H(\MA)^n_{y}=\bigoplus\limits_{n\in\ZZ}\ker((d_{\MA}^{x,y})^n)/\operatorname{im}((d_{\MA}^{x,y})^{n-1}
\]

We will denote \[
\bar{\MO}=\bigsqcup_{n\in\NN^*}\MO^n
\]and more generally, we will denote by $\doubar{\MO}$ the set formed by all finite tuples of elements of $\bar{\MO}$, \textit{i.e.} 
\begin{equation}
\begin{split}
\bar{\bar{\MO}}=\bigsqcup_{n\in\NN^*} \bar{\MO}^n=\bigsqcup_{n>0}\bigsqcup_{(p_1,\dots,p_n)\in\mathcal{T}_n}\MO^{p_1}\times\dots\times \MO^{p_n}
\end{split}
\end{equation}  
where $\mathcal{T}_n=\NN^n$ for $n>1$ and $\mathcal{T}_1=\NN^*$.
Given $\bar{x}=(x_1,\dots,x_n)\in\bar{\MO}$ we define its \textbf{\textcolor{ultramarine}{length}} as $\llg(\bar{x})=n$, its \textbf{\textcolor{ultramarine}{left term}} as $\llt(\bar{x})=x_1$ and \textbf{\textcolor{ultramarine}{right term}} as $\rrt(\bar{x})=x_n$. For $i\in\llbracket 1,n\rrbracket$, we define $\bar{x}_{\leq i}=(x_1,\dots,x_i)$, $\bar{x}_{\geq i}=(x_i,\dots,x_n)$ and for $j>i$, $\bar{x}_{\llbracket i,j\rrbracket}=(x_i,x_{i+1},\dots,x_j)$. One can similarly define $\bar{x}_{<i}$ and $\bar{x}_{>i}$.
Moreover, given $\doubar{x}=(\bar{x}^1,\dots,\bar{x}^n)\in\doubar{\MO}$ we define its \textbf{\textcolor{ultramarine}{length}} as $\llg(\doubar{x})=n$, its \textbf{\textcolor{ultramarine}{size}} as $N(\doubar{x})=\sum_{1\leq i\leq n}\llg(\bar{x}^i)$ its \textbf{\textcolor{ultramarine}{left term}} as $\llt(\doubar{x})=\bar{x}^1$ and its \textbf{\textcolor{ultramarine}{right term}} as $\rrt(\doubar{x})=\bar{x}^n$. 
For $\bar{x}=(x_1,\dots,x_n)\in\bar{\MO}$, we will denote 
\[
\MA^{\otimes \bar{x}}={}_{x_1}\MA_{x_2}\otimes {}_{x_2}\MA_{x_3}\otimes\dots\otimes {}_{x_{\llg(\bar{x})-1}}\MA_{x_{\llg(\bar{x})}}
\]
and we will often denote an element of $\MA^{\otimes\bar{x}}$ as $a_1,a_2,\dots,a_{\llg(\bar{x})-1}$ instead of $a_1\otimes a_2\otimes \dots \otimes a_{\llg(\bar{x})-1}$ for $a_i\in {}_{x_i}\MA_{x_{i+1}}$, $i\in\llbracket 1,\llg(\bar{x})-1\rrbracket$. Moreover, given a tuple $\doubar{x}=(\bar{x}^1,\dots,\bar{x}^n)\in\doubar{\MO}$ we will denote $\MA^{\otimes \doubar{x}}=\MA^{\otimes \bar{x}^1}\otimes \MA^{\otimes \bar{x}^2}\otimes\dots\otimes \MA^{\otimes \bar{x}^n}$.
Given tuples $\bar{x}=(x_1,\dots,x_n),\bar{y}=(y_1,\dots,y_m)\in\bar{\MO}$, we define their concatenation as $\bar{x}\sqcup\bar{y}=(x_1,\dots,x_n,y_1,\dots,y_m)$.
We also define the inverse of a tuple $\bar{x}=(x_1,\dots,x_n)\in\bar{\MO}$ as $\bar{x}^{-1}=(x_n,x_{n-1},\dots,x_1)$.
If $\sigma\in\mathcalboondox{S}_n$ and $\bar{x}=(x_1,\dots,x_n)\in\MO^n$, we define $\bar{x}\cdot\sigma=(x_{\sigma(1)},x_{\sigma(2)},\dots,x_{\sigma(n)})$.
Moreover, given $\sigma\in\mathcalboondox{S}_n$ and $\doubar{x}=(\bar{x}^1,\dots,\bar{x}^n)\in\bar{\MO}^n$, we define $\doubar{x}\cdot\sigma=(\bar{x}^{\sigma(1)},\bar{x}^{\sigma(2)},\dots,\bar{x}^{\sigma(n)})$.
We denote by $C_n$ the subgroup of $\mathcalboondox{S}_n$ generated by the cycle $\sigma=(1 2\dots n)$ which sends $i\in \llbracket 1,n-1\rrbracket $ to $i+1$ and $n$ to $1$.

In this paper, we will often make the use of the graphical calculus introduced in \cite{ktv}. 
We refer the reader to \cite{moi} for a detailed account of the involved definitions and terminology which we shall follow in this article. Note that given a diagram, we will often omit the corresponding bold arrow, 
\textit{i.e.} the bold arrow can be any of the outgoing arrows of the diagram.

\section{\texorpdfstring{Transferring $A_{\infty}$-structures}
{Transferring-A-infinity-structures}}
\label{A-inf-case}
This section is devoted to recall the notions of $A_{\infty}$-categories and $A_{\infty}$-morphisms.
For convenience of the reader, we also present the results appearing in \cite{petersen} and we give proofs entirely in terms of the structure of the $A_{\infty}$-category, instead of the associated bar constructions, which is the case of \cite{petersen}.
This is merely done in order to allow the comparison with the case of pre-Calabi-Yau categories appearing in the next section. We do not claim any originality in this regard. 
\begin{definition}
    Given graded quivers $\MA$ and $\MB$ with respective set of objects $\MO_{\MA}$ and $\MO_{\MB}$ as well as a map $\Phi_0 : \MO_{\MA}\rightarrow \MO_{\MB}$, we define the graded quiver
    \[
    C(\MA,\MB)=\prod\limits_{p\geq 1}\prod\limits_{\bar{x}\in\MO_{\MA}^p}C^{\bar{x}}(\MA,\MB)=\prod\limits_{p\geq 1}\prod\limits_{\bar{x}\in\MO_{\MA}^p}\Homgr_{\kk}(\MA[1]^{\otimes \bar{x}}, {}_{\Phi_0(\llt(\bar{x}))}\MB_{\Phi_0(\rrt(\bar{x}))})
    \]
    We will denote $C(\MA,\MA)$ simply by $C(\MA)$.
\end{definition}

\begin{remark}
    If $(\MA,d_{\MA})$ and $(\MB,d_{\MB})$ are dg quivers, $C(\MA,\MB)$ becomes a dg quiver whose differential is given by 
    $d_{C(\MA,\MB)}F^{\bar{x}}=d_{\MB}\circ F^{\bar{x}} -(-1)^{|\mathbf{F}|}F^{\bar{x}}\circ d_{\MA[1]^{\otimes \bar{x}}}$
    for a homogeneous element $\mathbf{F}\in C(\MA,\MB)$.
\end{remark}

\begin{definition}
\label{def:A-inf-alg}
An \textbf{\textcolor{ultramarine}{$A_{\infty}$-structure}} on a graded quiver $\MA$ with set of objects $\MO$ is an element of degree $1$ $sm_{\MA}\in C(\MA,\MB)[1]$ satisfying the Stasheff identities introduced by J. Stasheff in \cite{stasheff} and given by
\begin{equation}
    \label{eq:stasheff-identities}
   \tag{$\operatorname{SI}^{\bar{x}}$}
    \sum\limits_{1\leq i <j \leq n}sm_{\MA}^{\bar{x}_{\leq i}\sqcup\bar{x}_{\geq j}}\circ(\id^{\otimes \bar{x}_{\leq i}}\otimes sm_{\MA}^{\bar{x}_{\llbracket i,j\rrbracket}}\otimes \id^{\otimes \bar{x}_{\geq j}})=0
\end{equation}
for each $n\in\NN^*$ and $\bar{x}\in\MO^n$. 
This is tantamount to $[sm_{\MA},sm_{\MA}]_{G}=0$ where $[-,-]_G$ denotes the Gerstenhaber bracket.
In terms of diagrams, this reads $\sum\mathcal{E}(\mathcalboondox{D})=0$ where the sum is over all the filled diagrams $\mathcalboondox{D}$ of type $\bar{x}$ and of the form 



 A graded quiver endowed with an $A_{\infty}$-structure is called an \textbf{\textcolor{ultramarine}{$A_{\infty}$-category}}. 
\end{definition}

\begin{remark}
\label{remark:A-inf-dg}
Note that the identity $(SI^{x,y})$ is simply $sm_{\MA}^{x,y}\circ sm_{\MA}^{x,y}=0$, meaning that an $A_{\infty}$-category $(\MA,sm_{\MA})$ gives rise to a dg quiver $(\MA,(sm_{\MA}^{x,y}[-1])_{x,y\in\MO})$, whose cohomology will be denoted by $H(\MA)$. 
\end{remark}

The following definition is due to M. Sugawara (see \cite{sugawara}).

\begin{definition}
An $A_{\infty}$\textbf{\textcolor{ultramarine}{-morphism}} between $A_{\infty}$-categories $(\MA,sm_{\MA})$ and $(\MB,sm_{\MB})$ with respective sets of objects $\MO_{\MA}$ and $\MO_{\MB}$ is a map $F_0 : \MO_{\MA}\rightarrow \MO_{\MB}$ together with an element $s\mathbf{F}=(sF^{\bar{x}})_{\bar{x}\in\bar{\MO}_{\MA}}\in C(\MA,\MB)[1]$ of degree 0, that satisfies
\begin{equation}
\label{eq:stasheff-morphisms}
\tag{$\operatorname{MI}^{\bar{x}}$}
\begin{split}
    \sum\limits_{1\leq i<j\leq \llg(\bar{x})}sF^{\bar{x}_{\leq i}\sqcup\bar{x}_{\geq j}}&\circ(\id^{\otimes \bar{x}_{\leq i}}\otimes sm_{\MA}^{\bar{x}_ {\llbracket i,j\rrbracket}}\otimes \id^{\otimes \bar{x}_{\geq j}})\\&=\sum\limits_{1\leq i_1<\dots<i_n\leq \llg(\bar{x})}sm_{\MB}^{\bar{y}}(sF^{\bar{x}_{\leq i_1}}\otimes sF^{\bar{x}_{\llbracket i_1,i_2\rrbracket}}\otimes\dots\otimes sF^{\bar{x}_{\llbracket i_n,\llg(\bar{x})\rrbracket}})
    \end{split}
\end{equation}
for every $\bar{x}\in\bar{\MO}_{\MA}$ and with $\bar{y}=(F_0(x_1),F_0(x_{i_1}),\dots,F_0(x_{\llg(\bar{x})}))$. This is tantamount to 
\[
s\mathbf{F}\upperset{G}{\circ} sm_{\MA}=sm_{\MB}\upperset{M}{\circ} s\mathbf{F}
\]
where the composition $(sm_{\MB}\upperset{M}{\circ}s\mathbf{F})^{\bar{x}}$ is the right hand side of the identity \eqref{eq:stasheff-morphisms}. 
Note that given $\bar{x}\in\MO_{\MA}$ the \\\vskip-4mm\noindent terms in both sums are sums of maps associated with diagrams of type $\bar{x}$ that are of the form 

\begin{minipage}{21cm}


\end{minipage}

\noindent respectively.
\end{definition}

\begin{definition}
    Let $(\MA,sm_{\MA})$, $(\MB,sm_{\MB})$ and $(\mathcal{C},sm_{\mathcal{C}})$ be $A_{\infty}$-categories and consider $A_{\infty}$-morphisms $(F_0,s\mathbf{F}) : (\MA,sm_{\MA})\rightarrow (\MB,sm_{\MB})$ and $(G_0,s\mathbf{G}) : (\MB,sm_{\MB})\rightarrow (\mathcal{C},sm_{\mathcal{C}})$. The \textbf{\textcolor{ultramarine}{composition of $(F_0,s\mathbf{F})$ and $(G_0,s\mathbf{G})$}} is the pair $(G_0\circ F_0,s\mathbf{G}\circ s\mathbf{F})$ where the element $(s\mathbf{G}\circ s\mathbf{F})\in C(\MA,\mathcal{C})[1]$ is defined by
    \begin{equation}
        \label{eq:compo-Ainf}
        (s\mathbf{G}\circ s\mathbf{F})^{\bar{x}}=\sum\limits_{1\leq i_1<\dots<i_n\leq \llg(\bar{x})}sG^{\bar{y}}(sF^{\bar{x}_{\leq i_1}}\otimes sF^{\bar{x}_{\llbracket i_1,i_2\rrbracket}}\otimes\dots\otimes sF^{\bar{x}_{\llbracket i_n,\llg(\bar{x})\rrbracket}})
    \end{equation}
    for every $\bar{x}\in\bar{\MO}_{\MA}$ and with $\bar{y}=(F_0(x_1),F_0(x_{i_1}),\dots,F_0(x_{\llg(\bar{x})}))$.
\end{definition}

\begin{lemma}
    The data of $A_{\infty}$-categories together with $A_{\infty}$-morphisms and their composition give rise to a category denoted $A_{\infty}$.
    Moreover, given an $A_{\infty}$-category $(\MA,sm_{\MA})$, the \textbf{\textcolor{ultramarine}{$A_{\infty}$-identity morphism on $\MA$}} is the pair $(Id_{\MA},\id_{\MA})$ where $Id_{\MA}$ is the identity map on the set of objects of $\MA$ and $\id_{\MA}$ is the $A_{\infty}$-morphism whose component $\bar{x}$ is $\id_{{}_{x}\MA_{y}[1]}$ if $\bar{x}=(x,y)$ and $0$ if $\llg(\bar{x})>2$. We will denote the pair $(Id_{\MA},\id_{\MA})$ by $(Id,\id)$ when there is no possible confusion.
\end{lemma}

\begin{definition}
      Consider $A_{\infty}$-categories $(\MA,sm_{\MA})$ and $(\MB,sm_{\MB})$ and an $A_{\infty}$-morphism between them $(F_0,s\mathbf{F}):(\MA,sm_{\MA})\rightarrow (\MB,sm_{\MB})$. We say that $(F_0,s\mathbf{F})$ is a \textbf{\textcolor{ultramarine}{quasi-isomorphism}} if $F_0$ is a bijection and $sF^{x,y} : {}_{x}A_{y}[1]\rightarrow {}_{x}B_{y}[1]$ is a quasi-isomorphism of dg vector spaces for every $x,y\in\MO_{\MA}$, where the differentials of $\MA$ and $\MB$ are induced by their $A_{\infty}$-structures as in Remark \ref{remark:A-inf-dg}.
    Moreover, $(F_0,s\mathbf{F})$ is an \textbf{\textcolor{ultramarine}{isomorphism}} if there exists an $A_{\infty}$-morphism $(G_0,s\mathbf {G}):\MB\rightarrow \MA$ such that $G_0\circ F_0=Id_{\MA}$, $F_0\circ G_0=Id_{\MB}$, $s\mathbf{G}\circ s\mathbf{F}=\id_{\MA}$ and $s\mathbf{F}\circ s\mathbf{G}=\id_{\MB}$.
\end{definition}

We now recall the following result and give a proof of it without using the bar construction.

\begin{lemma}
\label{lemma:iso-A-inf}
    Let $(\MA,sm_{\MA})$, $(\MB,sm_{\MB})$ be $A_{\infty}$-categories and $(F_0,s\mathbf{F}) : (\MA,sm_{\MA})\rightarrow (\MB,sm_{\MB})$ be an $A_{\infty}$-morphism.
    Suppose that $F_0$ is a bijection and that $sF^{x,y}$ is an isomorphism of dg vector spaces for every $x,y\in\MO_{\MA}$. 
    Then, $(F_0,s\mathbf{F})$ is an isomorphism of $A_{\infty}$-categories.
\end{lemma}
\begin{proof}
Since $F_0$ is a bijection, it has an inverse $G_0 : \MO_{\MB}\rightarrow \MO_{\MA}$.
    Similarly, since $sF^{x,y}$ is an isomorphism of dg vector spaces for every $x,y\in\MO_{\MA}$ it has an inverse that we denote by $sG^{F_0(x),F_0(y)}$. Note that the bijectivity of $F_0$ allows us to define $sG^{x,y}$ for each $x,y\in\MO_{\MB}$.
    Consider $\bar{x}\in\bar{\MO}_{\MB}$ with $\llg(\bar{x})>2$ and suppose that we have constructed maps $sG^{\bar{y}} : \MB[1]^{\otimes \bar{y}}\rightarrow {}_{G_0(\llt(\bar{x}))}\MA_{G_0(\rrt(\bar{x}))}[1]$ satisfying that $(s\mathbf{G}\circ s\mathbf{F})^{\bar{y}}=0$ for every $\bar{y}\in\bar{\MO}_{\MB}$ with $2<\llg(\bar{y})<\llg(\bar{x})$.
   We define $sG^{\bar{x}}=-\sum\mathcal{E}(\mathcalboondox{D})$ where the sum is over all the filled diagrams $\mathcalboondox{D}$ of type $\bar{x}$ and of the form



\noindent where the dotted lines represent discs with one input and one output filled with $\mathbf{G}$.
    Therefore, we have constructed an element $s\mathbf{G}=(sG^{\bar{x}})_{\bar{x}\in\bar{\MO}_{\MB}}\in C(\MB,\MA)[1]$ such that $s\mathbf{G}\circ s\mathbf{F}=\id_{\MA}$. 
    Note that we did not use the fact that $s\mathbf{F}$ satisfies the identities \eqref{eq:stasheff-morphisms} for every $\bar{x}\in\bar{\MO}_{\MA}$. 
    Thus, there exists an element $s\mathbf{H}\in C(\MA,\MB)[1]$ such that $s\mathbf{H}\circ s\mathbf{G}=\id_{\MB}$. Using the associativity of the composition, we get that $s\mathbf{H}=s\mathbf{F}$ and thus $s\mathbf{G}\circ s\mathbf{F}=\id_{\MA}$ and $s\mathbf{F} \circ s\mathbf{G}=\id_{\MB}$.
    
    It remains to show that $s\mathbf{G}$ satisfies the identities \eqref{eq:stasheff-morphisms} for every $\bar{x}\in\bar{\MO}_{\MB}$. 
    Consider $\bar{x}\in\bar{\MO}_{\MB}$ and suppose that $s\mathbf{G}$ satisfies the identities $(\operatorname{MI}^{\bar{y}})$ for all $\bar{y}\in\bar{\MO}_{\MB}$ such that $\llg(\bar{y})<\llg(\bar{x})$.
    Since $(F_0,s\mathbf{F})$ is an $A_{\infty}$-morphism, we have that $s\mathbf{F}\upperset{G}{\circ}sm_{\MA}=sm_{\MB}\upperset{M}{\circ}s\mathbf{F}$ and then $(s\mathbf{F}\upperset{G}{\circ}sm_{\MA})\upperset{M}{\circ}s\mathbf{G}=(sm_{\MB}\upperset{M}{\circ}s\mathbf{F})\upperset{M}{\circ}s\mathbf{G}$.  \\ \vskip-3.8mm\noindent Therefore, we
     have that $s\mathbf{G}\upperset{G}{\circ}\big((s\mathbf{F}\upperset{G}{\circ}sm_{\MA})\upperset{M}{\circ}s\mathbf{G}\big)=s\mathbf{G}\upperset{G}{\circ}\big((sm_{\MB}\upperset{M}{\circ}s\mathbf{F})\upperset{M}{\circ}s\mathbf{G}\big)$. It is clear that the right \\ \vskip-3.8mm\noindent member of this equality is $s\mathbf{G}\upperset{G}{\circ}sm_{\MB}$ since $s\mathbf{F}\circ s\mathbf{G}=\id$.
     
     We now have to show that $s\mathbf{G}\upperset{G}{\circ}\big((s\mathbf{F}\upperset{G}{\circ}sm_{\MA})\upperset{M}{\circ}s\mathbf{G}\big)=sm_{\MA}\upperset{M}{\circ}s\mathbf{G}$.
    First, note that we have that \\ \vskip-3.8mm\noindent $\big((s\mathbf{F}\upperset{G}{\circ}sm_{\MA})\upperset{M}{\circ}s\mathbf{G}\big)^{\bar{x}}=\sum\mathcal{E}(\mathcalboondox{D_1})$ where the sum is over all the filled diagrams $\mathcalboondox{D_1}$ of type $\bar{x}$ and of the form 


    
\noindent Therefore, $(s\mathbf{G}\upperset{G}{\circ}\big((s\mathbf{F}\upperset{G}{\circ}sm_{\MA})\upperset{M}{\circ}s\mathbf{G}\big))^{\bar{x}}=(sm_{\MA}\upperset{M}{\circ}s\mathbf{G})^{\bar{x}}+\sum\mathcal{E}(\mathcalboondox{D_2})$ where the sum is over all the \\ \vskip-3.8mm\noindent filled diagrams $\mathcalboondox{D_2}$ of type $\bar{x}$ and of the form 

%

\noindent where one of the two bigger discs filled with $\mathbf{F}$ and $\mathbf{G}$ has strictly more than one incoming arrow.
Moreover, $\sum\mathcal{E}(\mathcalboondox{D_2})=0$ since $s\mathbf{F}\circ s\mathbf{G}=\id_{\MB}$. Then, $s\mathbf{G}\upperset{G}{\circ}\big((s\mathbf{F}\upperset{G}{\circ}sm_{\MA})\upperset{M}{\circ}s\mathbf{G}\big)=sm_{\MA}\upperset{M}{\circ}s\mathbf{G}$ which \\ \vskip-3.8mm\noindent finishes the proof.
\end{proof}

\begin{definition}
    Let $\MA$ and $\MB$ be two graded quivers and consider $s\mathbf{F}\in C(\MA,\MB)[1]$. 
    Given elements $sm\in C(\MB)[1]$
    and $sh\in C(\MA,\MB)[1]$
    we define the \textbf{\textcolor{ultramarine}{composition of $sm$ and $sh$ with respect to $s\mathbf{F}$}} as the element $sm\upperset{\scriptscriptstyle{s\mathbf{F}}}{\circ} sh\in C(\MA,\MB)[1]$
   given by $(sm\upperset{\scriptscriptstyle{s\mathbf{F}}}{\circ} sh)^{\bar{x}}=\sum\mathcal{E}(\mathcalboondox{D})$ where the sum is over all the filled diagrams \\\vskip-3.8mm\noindent $\mathcalboondox{D}$ of type $\bar{x}$ and of the form 


    
\end{definition}

\begin{lemma}
\label{lemma:identities}
 Let $\MA$ and $\MB$ be two graded quivers.
Consider $sm_{\MA}\in C(\MA)[1]$ and $sm_{\MB}\in C(\MB)[1]$ of degree $1$ and $s\mathbf{F}\in C(\MA,\MB)[1]$ of degree $0$.
Then, the following identities hold.
\begin{equation}
    \label{eq:associativity-mA}
    sm_{\MA}\upperset{\scriptscriptstyle{G}}{\circ}(sm_{\MA}\upperset{\scriptscriptstyle{G}}{\circ}sm_{\MA})=(sm_{\MA}\upperset{\scriptscriptstyle{G}}{\circ}sm_{\MA})\upperset{\scriptscriptstyle{G}}{\circ}sm_{\MA}
\end{equation}
\begin{equation}
    \label{eq:associativity-F-mB}
    (sm_{\MB}\upperset{\scriptscriptstyle{M}}{\circ}s\mathbf{F})\upperset{\scriptscriptstyle{G}}{\circ}sm_{\MA}=sm_{\MB}\upperset{\scriptscriptstyle{s\mathbf{F}}}{\circ}(s\mathbf{F}\upperset{\scriptscriptstyle{G}}{\circ}sm_{\MA})
\end{equation}
\end{lemma}
\begin{proof}
 For every $\bar{x}\in\bar{\MO}_{\MA}$, we have that 
 \[
 (sm_{\MA}\upperset{\scriptscriptstyle{G}}{\circ}(sm_{\MA}\upperset{\scriptscriptstyle{G}}{\circ}sm_{\MA}))^{\bar{x}}=\sum\mathcal{E}(\mathcalboondox{D})
 \] 
 where the sum is over all the filled diagrams $\mathcalboondox{D}$ of type $\bar{x}$ and of the form



    On the other hand, $((sm_{\MA}\upperset{\scriptscriptstyle{G}}{\circ}sm_{\MA})\upperset{\scriptscriptstyle{G}}{\circ}sm_{\MA})^{\bar{x}}=\sum\mathcal{E}(\mathcalboondox{D_1})+\sum\mathcal{E}(\mathcalboondox{D_2})-\sum\mathcal{E}(\mathcalboondox{D_3})$ where the sum \\\vskip-3.8mm\noindent is over all the filled diagrams $\mathcalboondox{D_1}$, $\mathcalboondox{D_2}$ and $\mathcalboondox{D_3}$ of type $\bar{x}$ and of the form

    \begin{minipage}{21cm}
         %
    \end{minipage}
\noindent respectively.
    The minus sign comes from the fact that the order of the first outgoing arrows of the two sources filled with $m_{\MA}$ is reversed and $|sm_{\MA}|=1$. Thus, the identity \eqref{eq:associativity-mA} is satisfied since $\sum\mathcal{E}(\mathcalboondox{D_2})=\sum\mathcal{E}(\mathcalboondox{D_3})$ and $\sum\mathcal{E}(\mathcalboondox{D})=\sum\mathcal{E}(\mathcalboondox{D_1})$.

    We now show \eqref{eq:associativity-F-mB}. For $\bar{x}\in\bar{\MO}_{\MA}$ we have that $(sm_{\MB}\upperset{\scriptscriptstyle{s\mathbf{F}}}{\circ}(s\mathbf{F}\upperset{\scriptscriptstyle{G}}{\circ}sm_{\MA}))^{\bar{x}}=\sum\mathcal{E}(\mathcalboondox{D})$ where the sum \\ \vskip-3.8mm\noindent is over all the filled diagrams $\mathcalboondox{D}$ of type $\bar{x}$ and of the form 



\noindent It is clear that $((sm_{\MB}\upperset{\scriptscriptstyle{M}}{\circ}s\mathbf{F})\upperset{\scriptscriptstyle{G}}{\circ}sm_{\MA})^{\bar{x}}=\sum\mathcal{E}(\mathcalboondox{D})$ as well.
\end{proof}

\begin{corollary}
 Let $\MA$ and $\MB$ be two graded quivers.
Consider $sm_{\MA}\in C(\MA)[1]$ and $sm_{\MB}\in C(\MB)[1]$ of degree $1$ and $s\mathbf{F}\in C(\MA,\MB)[1]$ of degree $0$.
Suppose that $sm_{\MB}$ satisfies the Stasheff identities \eqref{eq:stasheff-identities} for every $\bar{x}\in\bar{\MO}_{\MA}$.
Then, the following identity holds.
\begin{equation}
    \label{eq:second-identity}
    s\mathbf{F}\upperset{\scriptscriptstyle{G}}{\circ}(sm_{\MA}\upperset{\scriptscriptstyle{G}}{\circ}sm_{\MA})=(s\mathbf{F}\upperset{\scriptscriptstyle{G}}{\circ}sm_{\MA}-sm_{\MB}\upperset{\scriptscriptstyle{M}}{\circ} s\mathbf{F})\upperset{\scriptscriptstyle{G}}{\circ} sm_{\MA}+sm_{\MB}\upperset{\scriptscriptstyle{s\mathbf{F}}}{\circ}(s\mathbf{F}\upperset{\scriptscriptstyle{G}}{\circ}sm_{\MA}-sm_{\MB}\upperset{\scriptscriptstyle{M}}{\circ} s\mathbf{F})
\end{equation}
\end{corollary}
\begin{proof}
Rewrite the equation \eqref{eq:second-identity} as
    \begin{equation}
    s\mathbf{F}\upperset{G}{\circ}(sm_{\MA}\upperset{G}{\circ}sm_{\MA})=(s\mathbf{F}\upperset{G}{\circ}sm_{\MA})\upperset{G}{\circ} sm_{\MA}-(sm_{\MB}\upperset{M}{\circ} s\mathbf{F})\upperset{G}{\circ} sm_{\MA}+sm_{\MB}\upperset{s\mathbf{F}}{\circ}(s\mathbf{F}\upperset{G}{\circ}sm_{\MA})-sm_{\MB}\upperset{s\mathbf{F}}{\circ}(sm_{\MB}\upperset{M}{\circ} s\mathbf{F})
\end{equation}
By Lemma \ref{lemma:identities}, we already have that 
$(sm_{\MB}\upperset{\scriptscriptstyle{M}}{\circ} s\mathbf{F})\upperset{\scriptscriptstyle{G}}{\circ} sm_{\MA}=sm_{\MB}\upperset{\scriptscriptstyle{s\mathbf{F}}}{\circ}(s\mathbf{F}\upperset{\scriptscriptstyle{G}}{\circ}sm_{\MA})$.
Moreover, one can \\ \vskip-3.8mm\noindent prove in the same way as we did for \eqref{eq:associativity-mA} that $s\mathbf{F}\upperset{\scriptscriptstyle{G}}{\circ}(sm_{\MA}\upperset{\scriptscriptstyle{G}}{\circ}sm_{\MA})=(s\mathbf{F}\upperset{\scriptscriptstyle{G}}{\circ}sm_{\MA})\upperset{\scriptscriptstyle{G}}{\circ} sm_{\MA}$.
Finally, since \\ \vskip-3.8mm\noindent $sm_{\MB}$ satisfies the identities \eqref{eq:stasheff-identities} for every $\bar{x}\in\bar{\MO}_{\MB}$, $sm_{\MB}\upperset{\scriptscriptstyle{\mathbf{F}}}{\circ}(sm_{\MB}\upperset{\scriptscriptstyle{M}}{\circ} s\mathbf{F})=(sm_{\MB}\upperset{\scriptscriptstyle{G}}{\circ}sm_{\MB})\upperset{\scriptscriptstyle{\mathbf{F}}}{\circ} s\mathbf{F}=0$.
\end{proof}

We now recall the Homotopy Transfer Theorem written in terms of $A_{\infty}$-categories and give a proof which do not use the bar construction inspired by the one in \cite{petersen}.
\begin{theorem}
\label{thm2-Ainf}
Let $(\MA,d_{\MA})$, $(\MB,d_{\MB})$ be dg quivers and $(F_0,f):\MA\rightarrow \MB$ be a quasi-isomorphism of dg quivers.
Let $sm_{\MB}\in C(\MB)[1]$ be an element of degree $1$, with $sm_{\MB}^{x,y}=d_{\MB}^{x,y}[1]$ for every $x,y\in\MO_{\MB}$, that turns $\MB$ into an $A_{\infty}$-category.
Then, there exists an $A_{\infty}$-structure $sm_{\MA}$ on $\MA$ with $sm_{\MA}^{x,y}=d_{\MA}^{x,y}[1]$ as well as an $A_{\infty}$-morphism $(F_0,s\mathbf{F}) : (\MA,sm_{\MA})\rightarrow (\MB,sm_{\MB})$ such that $sF^{x,y}=f^{x,y}[1]$ for every $x,y\in\MO_{\MA}$. 
\end{theorem}
\begin{proof}
    We construct the $A_{\infty}$-structure and the $A_{\infty}$-morphism inductively.
    We set $sm_{\MA}^{x,y}=d_{\MA}^{x,y}[1]$ and $sF^{x,y}=f^{x,y}[1]$. Then, we already have that the identities $(\operatorname{SI}^{x,y})$ and $(\operatorname{MI}^{x,y})$ are satisfied for every $x,y\in\MO_{\MA}$.
    Consider $\bar{x}\in\bar{\MO}_{\MA}$ with $\llg(\bar{x})>2$  and suppose that we have constructed $sm_{\MA}\in C(\MA)[1]$ of degree $1$ and $s\mathbf{F}\in C(\MA,\MB)[1]$ of degree $0$ such that the identities $(\operatorname{SI}^{\bar{y}})$ and $(\operatorname{MI}^{\bar{y}})$ are satisfied for every $\bar{y}\in\bar{\MO}_{\MA}$ such that $\llg(\bar{y})<\llg(\bar{x})$.
    
    Since $(sm_{\MA}\upperset{\scriptscriptstyle{G}}{\circ} sm_{\MA})^{\bar{y}}$ vanishes for every $\bar{y}\in\bar{\MO}_{\MA}$ such that $\llg(\bar{y})<\llg(\bar{x})$ by assumption, the term \\ \vskip-3.8mm \noindent $\bar{x}$ of the identity \eqref{eq:associativity-mA} is simply 
    \[d_{\MA[1]}\circ (sm_{\MA}\upperset{\scriptscriptstyle{G}}{\circ} sm_{\MA})^{\bar{x}}=(sm_{\MA}\upperset{\scriptscriptstyle{G}}{\circ} sm_{\MA})^{\bar{x}}\circ d_{\MA[1]^{\otimes \bar{x}}}\] 
    which gives $d_{C(\MA)[1]}(sm_{\MA}\upperset{\scriptscriptstyle{G}}{\circ} sm_{\MA})^{\bar{x}}=0$ \textit{i.e.} $(sm_{\MA}\upperset{\scriptscriptstyle{G}}{\circ} sm_{\MA})^{\bar{x}}$ is a cycle.

    Moreover, since $(s\mathbf{F}\upperset{\scriptscriptstyle{G}}{\circ}sm_{\MA}-sm_{\MB}\upperset{\scriptscriptstyle{M}}{\circ} s\mathbf{F})^{\bar{y}}$ vanishes for $\bar{y}\in\bar{\MO}_{\MA}$ such that $\llg(\bar{y})<\llg(\bar{x})$, the term \\ \vskip-3.8mm \noindent $\bar{x}$ of the identity \eqref{eq:second-identity} is 
   \begin{equation}
    sF^{\llt(\bar{x}),\rrt(\bar{x})}\upperset{}{\circ}(sm_{\MA}\upperset{\scriptscriptstyle{G}}{\circ}sm_{\MA})^{\bar{x}}=(s\mathbf{F}\upperset{\scriptscriptstyle{G}}{\circ}sm_{\MA}-sm_{\MB}\upperset{\scriptscriptstyle{M}}{\circ} s\mathbf{F})^{\bar{x}}\upperset{\scriptscriptstyle{M}}{\circ} d_{\MA[1]^{\otimes \bar{x}}}+d_{\MB[1]}\underset{}{\circ}(s\mathbf{F}\upperset{\scriptscriptstyle{G}}{\circ}sm_{\MA}-sm_{\MB}\upperset{\scriptscriptstyle{M}}{\circ} s\mathbf{F})^{\bar{x}}
\end{equation}
so $sF^{\llt(\bar{x}),\rrt(\bar{x})}\upperset{}{\circ}(sm_{\MA}\upperset{\scriptscriptstyle{G}}{\circ}sm_{\MA})^{\bar{x}}$ is a boundary. Since $(sm_{\MA}\upperset{\scriptscriptstyle{G}}{\circ}sm_{\MA})^{\bar{x}}$ is a cycle and $(F_0,s\mathbf{F})$ is a \\ \vskip-3.8mm \noindent quasi-isomorphism, $(sm_{\MA}\upperset{\scriptscriptstyle{G}}{\circ}sm_{\MA})^{\bar{x}}$ is itself a boundary, meaning that there exists $e\in C^{\bar{x}}(\MA)[1]$ such \\ \vskip-3.8mm \noindent that $(sm_{\MA}\upperset{\scriptscriptstyle{G}}{\circ}sm_{\MA})^{\bar{x}}=\partial_{\MA} e$ where $\partial_{\MA}$ denotes the differential of $C(\MA)[1]$. 
We thus define the element \\ \vskip-3.8mm \noindent $\mu_{\MA}\in C(\MA)$ by $s\mu_A^{\bar{y}}=sm_{\MA}^{\bar{y}}$ for every $\bar{y}\in\bar{\MO}_{\MA}$ such that $\llg(\bar{y})<\llg(\bar{x})$ and $s\mu_{\MA}^{\bar{x}}=sm_{\MA}^{\bar{x}}-e$. 
It satisfies that $(s\mu_{\MA}\upperset{\scriptscriptstyle{G}}{\circ}s\mu_{\MA})^{\bar{y}}=0$ for every $\bar{y}\in\bar{\MO}_{\MA}$ with $\llg(\bar{y})\leq \llg(\bar{x})$.

Moreover, by \eqref{eq:second-identity} we have that 
$(s\mathbf{F}\upperset{\scriptscriptstyle{G}}{\circ}sm_{\MA} -sm_{\MB}\upperset{\scriptscriptstyle{M}}{\circ}s\mathbf{F})^{\bar{x}}$ is a cycle so it can be written as 
\[
(s\mathbf{F}\upperset{\scriptscriptstyle{G}}{\circ}sm_{\MA} -sm_{\MB}\upperset{\scriptscriptstyle{M}}{\circ}s\mathbf{F})^{\bar{x}}=sF^{\llt(\bar{x}),\rrt(\bar{x})}\circ e'+ \partial e\dprime
\]
where $e'$ is a cycle of $C^{\bar{x}}(\MA)[1]$, $e\dprime\in C^{\bar{x}}(\MA,\MB)[1]$ and where $\partial$ denotes the differential of $C(\MA,\MB)[1]$.

We thus define $s\nu_{\MA}^{\bar{y}}=s\mu_{\MA}^{\bar{y}}$, $sG^{\bar{y}}=sF^{\bar{y}}$ for every $\bar{y}\in\bar{\MO}_{\MA}$ such that $2<\llg(\bar{y})< \llg(\bar{x})$ and $s\nu_{\MA}^{\bar{x}}=s\mu_{\MA}^{\bar{x}}-e'$, $sG^{\bar{x}}=sF^{\bar{x}}-e\dprime$.
Therefore, since we have only modified $s\mu_{\MA}$ by adding a cycle, we have $(s\nu_{\MA}\upperset{\scriptscriptstyle{G}}{\circ}s\nu_{\MA})^{\bar{y}}=0$ for every $\bar{y}\in\bar{\MO}_{\MA}$ with $\llg(\bar{y})\leq \llg(\bar{x})$.
Moreover, it is straightforward to\\ \vskip-3.8mm \noindent check that $(G_0,s\mathbf{G})$ is an $A_{\infty}$-morphism.
\end{proof}

As proved in \cite{petersen} we also have the statement in the other direction, namely, we have the following theorem whose proof is similar to the one of Theorem \ref{thm2-Ainf}: 
\begin{theorem}
\label{thm3-Ainf}
    Let $(\MA,d_{\MA})$, $(\MB,d_{\MB})$ be dg quivers and $(F_0,f):\MA\rightarrow \MB$ be a quasi-isomorphism of dg quivers.
Let $sm_{\MA}\in C(\MA)[1]$ be an element of degree $1$, with $sm_{\MA}^{x,y}=d_{\MA}^{x,y}[1]$ for every $x,y\in\MO_{\MA}$, that turns $\MA$ into an $A_{\infty}$-category.
Then, there exists an $A_{\infty}$-structure $sm_{\MB}\in C(\MB)[1]$ on $\MB$ such that $sm_{\MB}^{x,y}=d_{\MB}^{x,y}[1]$ for every $x,y\in\MO_{\MB}$ as well as an $A_{\infty}$-morphism $(F_0,s\mathbf{F})$ with $sF\in C(\MA,\MB)[1]$ such that $sF^{x,y}=f^{x,y}[1]$ for every $x,y\in\MO_{\MA}$. 
\end{theorem}

We now recall the definition of a minimal model of an $A_{\infty}$-category.

\begin{definition}
    An $A_{\infty}$-category $(\MB,sm_{\MB})$ is 
    \textbf{\textcolor{ultramarine}{minimal}} if $sm_{\MB}^{x,y} = 0$ for each $x,y\in\MO_{\MB}$. 
    Given an $A_{\infty}$-category $(\MA,sm_{\MA})$, a \textbf{\textcolor{ultramarine}{minimal model}} of it is a minimal $A_{\infty}$-category $(\MB,sm_{\MB})$ together with a quasi-isomorphism $(P_0,s\mathbf{P}) : \MA \rightarrow \MB$ of $A_{\infty}$-categories.
\end{definition}

We remark the well-known fact that minimal models of $A_{\infty}$-categories always exist. 
Indeed, given an $A_{\infty}$-category $(\MA,sm_{\MA})$,  
let $(Id,p) : \MA\rightarrow H(\MA)$ be any quasi-isomorphism of dg quivers. 
By Theorem \ref{thm3-Ainf}, there exists an $A_{\infty}$-structure on $H(\MA)$ and a morphism of $A_{\infty}$-categories $(Id,s\mathbf{P}) : \MA \rightarrow H(\MA)$ such that $sP^{x,y} = p^{x,y}[1]$ for each $x,y\in\MO_{\MA}$. Moreover, we recall the following result.

\begin{lemma}
     Consider an $A_{\infty}$-category $(\MA,sm_{\MA})$, $H(\MA)$ its cohomology and two quasi-isomorphisms of dg quivers $(Id,p) : \MA\rightarrow H(\MA)$ and $(Id,i) : H(\MA)\rightarrow \MA$. Then, the $A_{\infty}$-structures $sm_{i}$ and $sm_{p}$ on $H(\MA)$ given in theorems \ref{thm2-Ainf} and \ref{thm3-Ainf} are isomorphic.
     
     In particular, if $(Id,p),(Id,p') : \MA\rightarrow H(\MA)$ (resp. $(Id,i), (Id,i') : H(\MA)\rightarrow \MA$) are quasi-isomorphisms, the $A_{\infty}$-structures $sm_{p}$ and $sm_{p'}$ (resp. $sm_{i}$, $sm_{i'}$) on $H(\MA)$ given in Theorem \ref{thm3-Ainf} (resp. in Theorem \ref{thm2-Ainf}) are isomorphic.
\end{lemma}

\begin{definition}
\label{def:morphism-homology}
    Let $(\MA,sm_{\MA})$ and $(\MB,sm_{\MB})$ be $A_{\infty}$-categories. 
    We consider quasi-isomorphisms of dg quivers $(Id,i_{\MA}) : H(\MA)\rightarrow \MA$ and $(Id,p_{\MB}) : \MB\rightarrow H(\MB)$. Then, by Theorem \ref{thm2-Ainf} $H(\MA)$ is endowed with an $A_{\infty}$-structure $sm_{i_{\MA}}$ and $(Id,i_{\MA})$ extends to an $A_{\infty}$-morphism $(Id,s\mathbf{I}_{\MA}) : H(\MA)\rightarrow \MA$. Moreover, by Theorem \ref{thm3-Ainf} $H(\MB)$ is endowed with an $A_{\infty}$-structure $sm_{p_{\MB}}$ and $(Id,p_{\MB})$ extends to an $A_{\infty}$-morphism $(Id,s\mathbf{P}_{\MB}) : \MB\rightarrow H(\MB)$.
    Given an $A_{\infty}$-morphism $(F_0,s\mathbf{F}):\MA\rightarrow \MB$, we define an $A_{\infty}$-morphism $(F_0,H(s\mathbf{F})) : H(\MA)\rightarrow H(\MB)$ where $H(s\mathbf{F})=s\mathbf{P}_{\MB} \circ s\mathbf{F}\circ s\mathbf{I}_{\MA}$. 
\end{definition}
Using the previous results, we deduce the following well-known theorem appearing for example in \cite{lefevre}.
\begin{theorem}
    Let $(\MA,sm_{\MA})$, $(\MB,sm_{\MB})$ be two $A_{\infty}$-categories and $(F_0,s\mathbf{F}) : (\MA,sm_{\MA})\rightarrow (\MB,sm_{\MB})$ be an $A_{\infty}$-morphism. Consider quasi-iso\-mor\-phisms $(Id,i_{\MA}) : H(\MA)\rightarrow \MA$ and $(Id,p_{\MB}) : \MB\rightarrow H(\MB)$ and the associated $A_{\infty}$-morphism $(F_0,H(s\mathbf{F})) : (H(\MA),sm_{i_{\MA}})\rightarrow (H(\MB),sm_{p_{\MB}})$ given in Definition \ref{def:morphism-homology}.
    
    Then, if $(F_0,s\mathbf{F})$ is a quasi-iso\-mor\-phism, there exists $(G_0,s\mathbf{G}):(\MB,sm_{\MB})\rightarrow (\MA,sm_{\MA})$ a quasi-iso\-mor\-phism of $A_{\infty}$-categories and quasi-isomorphisms $(Id,i_{\MB}) : H(\MB)\rightarrow \MB$, $(Id,p_{\MA}) : \MA\rightarrow H(\MA)$ such that $p_{\MA}\circ i_{\MA}=\id_{H(\MA)}$, $p_{\MB}\circ i_{\MB}=\id_{H(\MB)}$ and whose associated $(G_0,H(s\mathbf{G}))$ in the sense of Definition \ref{def:morphism-homology} is the inverse of $(F_0,H(s\mathbf{F}))$.
\end{theorem}

\section{Transferring pre-Calabi-Yau structures}
\label{pCY-case}
\subsection{Pre-Calabi-Yau categories}
In this subsection, we recall the notions of pre-Calabi-Yau structures and morphisms introduced in \cite{ktv} and show that a $d$-pre-Calabi-Yau morphism whose first component is an isomorphism is invertible in the category of $d$-pre-Calabi-Yau categories.
\begin{definition}
Given graded quivers $\MA$ and $\MB$ with respective sets of objects $\MO_{\MA}$ and $\MO_{\MB}$ and a map $F_0 : \MO_{\MA}\rightarrow \MO_{\MB}$, we define the graded vector space 
\begin{small}
    \begin{equation}
    \begin{split}
\Multi^{\bullet}_d(\MA,\MB)=\prod_{n\in\NN^*}\prod_{\doubar{x}\in\bar{\MO}_{\MA}^n}\Homgr_{\kk}\bigg(\bigotimes\limits_{i=1}^n\MA[1]^{\otimes \bar{x}^i},\bigotimes\limits_{i=1}^{n-1}{}_{F_0(\llt(\bar{x}^i))}\MB_{F_0(\rrt(\bar{x}^{i+1}))}[-d]\otimes{}_{F_0(\llt(\bar{x}^n))}\MB_{F_0(\rrt(\bar{x}^1))}[-d])\bigg)
    \end{split}
    \end{equation}
\end{small}
The action of $\sigma=(\sigma_n)_{n\in\NN^*}\in\prod_{n\in\NN^*} C_n$ on an element $\mathbf{F} = (F^{\doubar{x}})_{\doubar{x} \in \doubar{\MO}}\in\Multi^{\bullet}_d(\MA,\MB)$ is the element $\sigma\cdot\mathbf{F}\in\Multi^{\bullet}_d(\MA,\MB)$ given by 
\begin{equation}
    \begin{split}
(\sigma\cdot\mathbf{F})^{\doubar{x}}=\tau^{\sigma^{-1}}_{{}_{F_0(\llt(\bar{x}^{1}))}\MB_{F_0(\rrt(\bar{x}^2))}[-d],\dots, {}_{F_0(\llt(\bar{x}^n))}\MB_{F_0(\rrt(\bar{x}^1))}[-d]}\circ F^{\doubar{x}\cdot\sigma} \circ \tau^{\sigma}_{\MA[1]^{\otimes \bar{x}^1},\MA[1]^{\otimes\bar{x}^2}, \dots ,\MA[1]^{\otimes \bar{x}^n}}
    \end{split}
\end{equation}
for $\doubar{x}\in\doubar{\MO}$.
We will denote by $\Multi^{\bullet}_d(\MA,\MB)^{C_{\llg(\bullet)}}$ the space of elements of $\Multi^{\bullet}_d(\MA,\MB)$ that are invariant under the action of $\prod_{n\in\NN^*} C_n$. Moreover, we will denote $\Multi^{\bullet}_d(\MA,\MA)$ simply by $\Multi^{\bullet}_d(\MA)$.
\end{definition}

\begin{remark}
    If $(\MA,d_{\MA})$ and $(\MB,d_{\MB})$ are dg quivers and $F_0 :\MA\rightarrow\MB$, $\Multi^{\bullet}_d(\MA,\MB)$ becomes a dg quiver, with differential given by 
    \[
    d_{\Multi^{\doubar{x}}_d(\MA,\MB)}F=d_{\bigotimes\limits_{i=1}^{n-1}{}_{F_0(\llt(\bar{x}^i))}\MB_{F_0(\rrt(\bar{x}^{i+1}))}[-d]\otimes{}_{F_0(\llt(\bar{x}^n))}\MB_{F_0(\rrt(\bar{x}^1))}[-d]}\circ F -(-1)^{|F|}F\circ d_{\MA[1]^{\otimes \doubar{x}}}\]
    for $F\in \Multi^{\doubar{x}}_d(\MA,\MB)$, $\doubar{x}=(\bar{x}^1,\dots,\bar{x}^n)\in\bar{\MO}_{\MA}^n$.
\end{remark}
\begin{definition}
A \textbf{\textcolor{ultramarine}{$d$-pre-Calabi-Yau structure}} on a graded quiver $\MA$ is an element $s_{d+1}M_{\MA}$ of degree $1$ in
$\Multi^{\bullet}_d(\MA)^{C_{\llg(\bullet)}}[d+1]$, solving the Maurer-Cartan equation 
\begin{equation}
    \label{eq:MC}
    [s_{d+1}M_{\MA},s_{d+1}M_{\MA}]_{\nec}=0
\end{equation}
where $[s_{d+1}M_{\MA},s_{d+1}M_{\MA}]_{\nec}$ is the \textbf{\textcolor{ultramarine}{necklace bracket}} defined in \cite{ktv} as the graded commutator of the necklace product $\upperset{\nec}{\circ}$.
This is tantamount to require that $s_{d+1}M_{\MA}\in \Multi_d(A)^{C_{\bullet}}[d+1]$ satisfies the \\\vskip-3.8mm\noindent identities $(\operatorname{SI}^{\doubar{x}})$ given by $\sum\mathcal{E}(\mathcalboondox{D})+\sum\mathcal{E}(\mathcalboondox{D'})=0$ where the sums are over all the filled diagrams $\mathcalboondox{D}$ and $\mathcalboondox{D'}$ of type $\doubar{x}$ and of the form

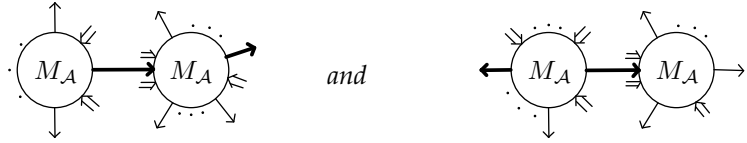
\begin{minipage}{21cm}

\end{minipage}

\noindent
    respectively, for every $\doubar{x}\in\doubar{\MO}$.
    A graded quiver endowed with a $d$-pre-Calabi-Yau structure is called a \textbf{\textcolor{ultramarine}{$d$-pre-Calabi-Yau category}}. 
\end{definition}

We now recall the definition of pre-Calabi-Yau morphisms as given in \cite{ktv}.
\begin{definition}
\label{def:pcY-morphism}
Given $d$-pre-Calabi-Yau categories $(\MA,s_{d+1}M_{\MA})$ and $(\MB,s_{d+1}M_{\MB})$ with respective sets of objects $\MO_{\MA}$ and $\MO_{\MB}$ a \textbf{\textcolor{ultramarine}{$d$-pre-Calabi-Yau morphism}} $(F_0,s_{d+1}\mathbf{F}) :(\MA,s_{d+1}M_{\MA}) \rightarrow (\MB,s_{d+1}M_{\MB})$ is a map $F_0 : \MO_{\MA}\rightarrow \MO_{\MB}$ together with an element $s_{d+1}\mathbf{F}\in \Multi^{\bullet}_d(\MA,\MB)^{C_{\llg(\bullet)}}[d+1]$
of degree $0$ satisfying the following equation
\begin{equation}
\label{eq:morphism-pCY}
\tag{$\operatorname{MI}^{\doubar{x}}$}
(s_{d+1}\mathbf{F}\upperset{\multinec}{\circ} s_{d+1}M_{\MA})^{\doubar{x}}= (s_{d+1}M_{\MB}\upperset{\pre}{\circ} s_{d+1}\mathbf{F})^{\doubar{x}}
\end{equation}
for all $\doubar{x}\in\doubar{\MO}_{\MA}$ where $(s_{d+1}\mathbf{F}\upperset{\multinec}{\circ} s_{d+1}M_{\MA})^{\doubar{x}}=\sum\mathcal{E}(\mathcalboondox{D})$ and $(s_{d+1}M_{\MB}\upperset{\pre}{\circ} s_{d+1}\mathbf{F})^{\doubar{x}}=\sum\mathcal{E}(\mathcalboondox{D'})$ where \\\vskip-3.8mm\noindent the sums are over all the filled diagrams  $\mathcalboondox{D}$ and $\mathcalboondox{D'}$ of type $\doubar{x}$ and of the form

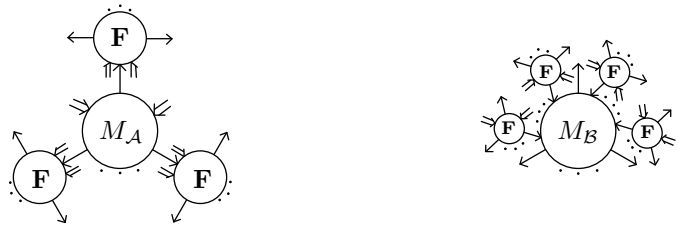
\begin{minipage}{21cm}

\end{minipage} 

\noindent respectively. Note that the left member and right member of the identity \eqref{eq:morphism-pCY} belong to
\[
\Homgr_{\kk}\bigg(\bigotimes\limits_{i=1}^{n}\MA[1]^{\otimes\bar{x}^{i}}, \bigotimes\limits_{i=1}^{n-1}{}_{F_0(\llt(\bar{x}^{i}))}\MB_{F_0(\rrt(\bar{x}^{i+1}))}[-d]\otimes{}_{F_0(\llt(\bar{x}^{n}))}\MB_{F_0(\rrt(\bar{x}^{1}))}[1]\bigg)
\]
if $\doubar{x}=(\bar{x}^1,\dots,\bar{x}^n)\in\bar{\MO}_{\MA}^n$.
\end{definition}

\begin{remark}
\label{remark:pCY-struct-dg}
    The identities $(\operatorname{SI}^{\bar{x}})$ for $\bar{x}\in\bar{\MO}_{\MA}$ are the same as the Stasheff identities given in Definition \ref{def:A-inf-alg}. Then, if $(\MA,s_{d+1}M_{\MA})$ is a $d$-pre-Calabi-Yau category the collection of maps $(s_{d+1}M_{\MA}^{\bar{x}})_{\bar{x}\in\bar{\MO}}$ turns $\MA$ into an $A_{\infty}$-category and as before, $(\MA,(s_{d+1}M_{\MA}^{x,y}[-1])_{x,y\in\MO_{\MA}})$ is a dg quiver. Moreover, a $d$-pre-Calabi-Yau morphism $(F_0,s\mathbf{F}):(\MA,s_{d+1}M_{\MA})\rightarrow(\MB,s_{d+1}M_{\MB})$ induces an $A_{\infty}$-morphism between the underlying $A_{\infty}$-categories.
\end{remark}

\begin{definition}
\label{def:cp-pCY}
Let $(\MA,s_{d+1}M_{\MA})$, $(\MB,s_{d+1}M_{\MB})$ and $(\mathcal{C},s_{d+1}M_{\mathcal{C}})$ be $d$-pre-Calabi-Yau categories with respective sets of objects $\MO_{\MA}$, $\MO_{\MB}$ and $\MO_{\mathcal{C}}$ and let $(F_0,s_{d+1}\mathbf{F}) :(\MA,s_{d+1}M_{\MA})\rightarrow (\MB,s_{d+1}M_{\MB})$ and $(G_0,s_{d+1}\mathbf{G}) :(\MB,s_{d+1}M_{\MB})\rightarrow (\mathcal{C},s_{d+1}M_{\mathcal{C}})$ be  $d$-pre-Calabi-Yau morphisms.
The \textbf{\textcolor{ultramarine}{composition of $s_{d+1}\mathbf{F}$ and $s_{d+1}\mathbf{G}$}} is the pair \[(G_0\circ F_0,s_{d+1}\mathbf{F}{\circ} s_{d+1}\mathbf{G})\] where
\[s_{d+1}\mathbf{F}{\circ} s_{d+1}\mathbf{G}\in \Multi^{\bullet}_d(\MA,\mathcal{C})^{C_{\llg(\bullet)}}[d+1]\] is given by $(s_{d+1}G{\circ} s_{d+1}F)^{\doubar{x}}=\sum\mathcal{E}(\mathcalboondox{D})$ where the sum is over all filled diagrams $\mathcalboondox{D}$ of type $\doubar{x}\in\doubar{\MO}_{\MA}$ and of the form
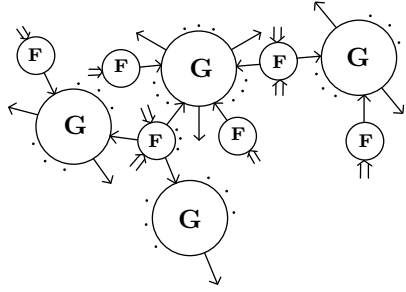
\begin{equation}

    \label{fig:composition}
\end{equation}
and where we have omitted the bold arrow, meaning that it any of the outgoing arrows. More precisely, the sum is over all the diagrams $\mathcalboondox{D}$ of type $\doubar{x}$ where the discs are filled with either $\mathbf{F}$ and $\mathbf{G}$ such that each outgoing arrows of a disc filled with $\mathbf{F}$ is connected to an incoming arrow of a disc filled with $\mathbf{G}$ and each incoming arrow of a disc filled with $\mathbf{G}$ is connected to an outgoing arrow of a disc filled with $\mathbf{F}$. In particular, the incoming (resp. outgoing) arrows of the diagram $\mathcalboondox{D}$ is an incoming (resp. outgoing) arrow of a disc filled with $\mathbf{F}$ (resp.$\mathbf{G}$).
\end{definition}

\begin{proposition}
For $d\in\ZZ$, $d$-pre-Calabi-Yau categories and $d$-pre-Calabi-Yau morphisms together with the composition given in Definition \ref{def:cp-pCY} define a category, denoted as $\operatorname{pCY_d}$. Given a graded quiver $\MA$ with set of objects $\MO$, the identity morphism $(Id_{\MA},\mathbf{\id}_{\MA}):(\MA,s_{d+1}M_{\MA})\rightarrow (\MA,s_{d+1}M_{\MA})$ is given by $\id_{\MA}^{x}=\id_{{}_{x}\MA_{x}}$ for $x\in\MO$ and $\id^
{\bar{x}^1,\dots,\bar{x}^n}=0$ for $(\bar{x}^1,\dots,\bar{x}^n)\in\bar{\MO}^n$ such that $n\neq 1$ or $n=1$ and $\llg(\bar{x}^1)\neq 1$.
We will denote the pair $(Id_{\MA},\mathbf{\id}_{\MA})$ by $(Id,\mathbf{\id})$ if there is no possible confusion.
\end{proposition}
\begin{proof}
    See for example \cite{moi} (cf. \cite{ktv,lv}).
\end{proof}

\begin{definition}
    Given d-pre-Calabi-Yau categories $(\MA,s_{d+1}M_{\MA})$ and $(\MB,s_{d+1}M_{\MB})$, we say that a d-pre-Calabi-Yau morphism $(F_0,s_{d+1}\mathbf{F}):(\MA,s_{d+1}M_{\MA})\rightarrow (\MB,s_{d+1}M_{\MB})$ is a \textbf{\textcolor{ultramarine}{quasi-isomorphism}} if $F_0$ is a bijection and the map $s_{d+1}F^{x,y} : {}_{x}\MA_{y}[1]\rightarrow {}_{x}\MB_{y}[1]$ is a quasi-isomorphism of dg vector spaces for every $x,y\in\MO_{\MA}$.
    Moreover, $s_{d+1}\mathbf{F}$ is an \textbf{\textcolor{ultramarine}{isomorphism}} if there exists a d-pre-Calabi-Yau morphism $(G_0,s_{d+1}\mathbf{G}):(\MB,s_{d+1}M_{\MB})\rightarrow (\MA,s_{d+1}M_{\MA})$ such that, $G_0\circ F_0=Id_{\MA}$, $F_0\circ G_0=Id_{\MB}$, $s_{d+1}\mathbf{G}\circ s_{d+1}\mathbf{F}=\mathbf{\id}_{\MA}$, and $s_{d+1}\mathbf{F}\circ s_{d+1}\mathbf{G}=\mathbf{\id}_{\MB}$.
\end{definition}

\begin{lemma}
\label{lemma:iso-pCY}
 Let $(\MA,s_{d+1}M_{\MA})$ and $(\MB,s_{d+1}M_{\MB})$ be $d$-pre-Calabi-Yau categories and consider a $d$-pre-Calabi-Yau morphism $(F_0,s_{d+1}\mathbf{F}):(\MA,s_{d+1}M_{\MA})\rightarrow (\MB,s_{d+1}M_{\MB})$ such that $F_0$ is a bijection. Then, if $s_{d+1}F^{x,y}$ is an isomorphism of dg vector spaces for every $x,y\in\MO_{\MA}$, $s_{d+1}\mathbf{F}$ is an isomorphism of $d$-pre-Calabi-Yau categories.
\end{lemma}
\begin{proof}
    By Lemma \ref{lemma:iso-A-inf}, we have already constructed maps $G^{\bar{x}} : \MB[1]^{\otimes \bar{x}}\rightarrow \MA[1]$ for every $\bar{x}\in\bar{\MO}_{\MB}$.
    Consider $\doubar{x}=(\bar{x}^1,\dots,\bar{x}^n)\in\bar{\MO}^n_{\MB}$ with $n>1$ and suppose that we have defined the maps \[G^{\doubar{y}} : \MB[1]^{\otimes \bar{y}^1}\otimes\dots\otimes \MB[1]^{\otimes \bar{y}^m}\rightarrow {}_{\llt(\bar{y}^1)}\MA_{\rrt(\bar{y}^2)}[-d]\otimes\dots\otimes {}_{\llt(\bar{y}^m)}\MA_{\rrt(\bar{y}^1)}[-d]\] satisfying that $(s_{d+1}G\circ s_{d+1}F)^{\doubar{y}}=0$ for every $\doubar{y}\in\doubar{\MO}_{\MB}$ with $1<\llg(\doubar{y})<\llg(\doubar{x})$ or $\llg(\doubar{y})=\llg(\doubar{x})$ and $N(\doubar{y})<N(\doubar{x})$ and that $(s_{d+1}F\circ s_{d+1}G)^{\doubar{z}}=0$ and $\doubar{z}\in\doubar{\MO}_{\MA}$ with $\llg(\doubar{z})<\llg(\doubar{x})$ or $\llg(\doubar{z})=\llg(\doubar{x})$ and $N(\doubar{z})<N(\doubar{x})$.
   
    We have that $(s_{d+1}\mathbf{G}\circ s_{d+1}\mathbf{F})^{\doubar{x}}=\mathcal{E}(\mathcalboondox{D_1})+\sum\mathcal{E}(\mathcalboondox{D_2})$ where $\mathcalboondox{D_1}$ is the unique diagram of type $\doubar{x}$ containing a disc of type $\doubar{x}$ filled with $\mathbf{G}$ whose incoming arrows are shared with discs with one incoming and one outgoing arrow filled with $\mathbf{F}$ and where the sum is over all the filled diagrams $\mathcalboondox{D_2}$ of type $\doubar{x}$ and of the form



\noindent
respectively. Note that $\mathcalboondox{D_1}$ is the only diagram containing a disc of type $\doubar{x}$ filled with $\mathbf{G}$. 
    Moreover, diagrams involved in the previous sum $\sum\mathcal{E}(\mathcalboondox{D_2})$ contain either only discs filled with $\mathbf{F}$ having one outgoing arrow or at least one disc filled with $\mathbf{F}$ with several outgoing arrows.
    In the first case, the terms only contain discs filled with $\mathbf{G}$ of type  $\doubar{y}\in\doubar{\MO}_{\MB}$ with $N(\doubar{y})<N(\doubar{x})$.
    Indeed, the discs filled with $\mathbf{F}$ only have one outgoing arrow but there is at least one of them having more that one incoming arrow. 
    In the second case, they only contain discs filled with $\mathbf{G}$ of type $\doubar{y}\in\doubar{\MO}_{\MB}$ with $\llg(\doubar{y})<\llg(\doubar{x})$ since the total number of outputs is $\llg(\doubar{x})$.
    This allows us to define $s_{d+1}G^{\doubar{x}}=-\sum\mathcal{E}(\mathcalboondox{D})$ where the sum is over all the filled diagrams $\mathcalboondox{D}$ of type $\doubar{x}$ and of the form



\noindent
where the dotted arrows represent discs with one incoming and one outgoing arrow filled with $\mathbf{G}$.
    Note that we did not use the fact that $s_{d+1}\mathbf{F}$ satisfies the identities \eqref{eq:morphism-pCY}. 
    Thus, there exists an element $s_{d+1}\mathbf{H}\in\Multi_d^{\bullet}(\MA,\MB)^{C_{\llg(\bullet)}}[d+1]$ such that $s_{d+1}\mathbf{H}\circ s_{d+1}\mathbf{G}=\id_{\MB}$. Using the associativity of the composition, we get that $\mathbf{H}=\mathbf{F}$ and thus $s_{d+1}\mathbf{G}\circ s_{d+1}\mathbf{F}=\id_{\MA}$ and $s_{d+1}\mathbf{F} \circ s_{d+1}\mathbf{G}=\id_{\MB}$.
    
    It remains to check that $(G_0,s_{d+1}\mathbf{G})$ is a $d$-pre-Calabi-Yau morphism. By Lemma \ref{lemma:iso-A-inf} the identities \eqref{eq:stasheff-morphisms} are satisfied for $\bar{x}\in\bar{\MO}_{\MB}$. We now consider $\doubar{x}\in\doubar{\MO}_{\MB}$ with $\llg(\doubar{x})>1$ and suppose that the identities $(\operatorname{MI}^{\doubar{y}})$ are satisfied for every $\doubar{y}\in\doubar{\MO}_{\MB}$ with $\llg(\doubar{y})<\llg(\doubar{x})$.
    Since $s_{d+1}\mathbf{F}$ satisfies the identities \eqref{eq:morphism-pCY}, for every $\doubar{x}\in\doubar{\MO}_{\MA}$ we have that $\sum\mathcal{E}(\mathcalboondox{D})=\sum\mathcal{E}(\mathcalboondox{D'})$ where the sums are over
    all the diagrams $\mathcalboondox{D}$ and $\mathcalboondox{D'}$ of type $\doubar{x}$ and of the form 

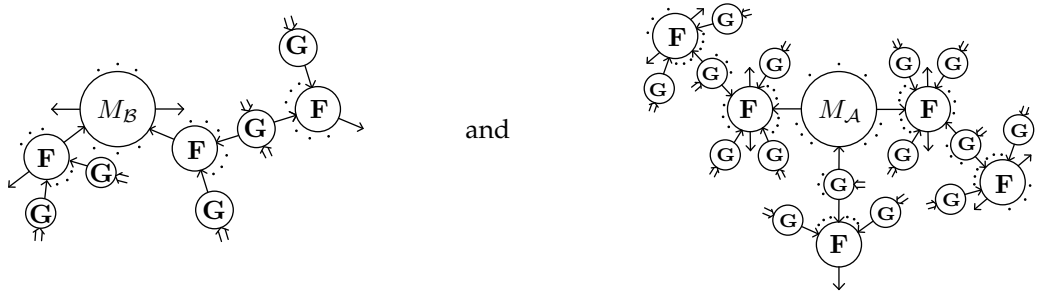
\begin{minipage}{21cm}

\end{minipage}
    \noindent respectively.
    Since $s_{d+1}\mathbf{F}\circ s_{d+1}\mathbf{G}=\id_{\MB}$, we have that $\sum\mathcal{E}(\mathcalboondox{D})=s_{d+1}M_{\MB}^{\doubar{x}}$. Therefore, we have that $(s_{d+1}\mathbf{G}\upperset{\multinec}{\circ}s_{d+1}M_{\MB})^{\doubar{x}}=\sum\limits\mathcal{E}(\mathcalboondox{D'_1})+\sum\mathcal{E}(\mathcalboondox{D'_2})$ where the sums are over all the diagrams of type $\doubar{x}$ and of the form 
    
    \begin{minipage}{21cm}
  %
\end{minipage}
    \noindent respectively.
    We now show that $\sum\mathcal{E}(\mathcalboondox{D'_2})=0$. Suppose that the identity $(\operatorname{MI}^{\doubar{y}})$ is also satisfied for every $\doubar{y}\in\doubar{\MO}_{\MB}$ with $\llg(\doubar{y})=\llg(\doubar{x})$ and $N(\doubar{y})<N(\doubar{x})$. We then have that $\sum\mathcal{E}(\mathcalboondox{D'_2})=\sum\mathcal{E}(\mathcalboondox{D'_3})$ where the second sum is over all the filled diagrams $\mathcalboondox{D'_3}$ of type $\doubar{x}$ and of the form
    
  %

    \noindent Then, by induction and since $(s_{d+1}\mathbf{F}\circ s_{d+1}\mathbf{G})^{\doubar{y}}=0$ for $\doubar{y}\neq(x,y)$ for $x,y\in\MO_{\MB}$, we have that $\sum\mathcal{E}(\mathcalboondox{D'_3})=0$ so $(G_0,s_{d+1}\mathbf{G})$ is a $d$-pre-Calabi-Yau morphism.
\end{proof}
\subsection{Homotopy Transfer Theorem for pre-Calabi-Yau categories}
In this subsection, we generalize the results of \cite{lv} to the case of $d$-pre-Calabi-Yau categories, giving direct proofs inspired by \cite{petersen}. We first define two different ways of composing maps and then give useful identities for the proofs of Theorem \ref{thm:2-pcy} and Theorem \ref{thm:3-pcy}.
\begin{definition}
Let $\MA$ and $\MB$ be two graded quivers and consider $s_{d+1}\mathbf{F}\in \Multi_d^{\bullet}(\MA,\MB)[d+1]$.
    Given $s_{d+1}M\in\Multi_d^{\bullet}(\MB)[d+1]$ and $s_{d+1}\mathbf{H}\in \Multi_d^{\bullet}(\MA,\MB)[d+1]$
    we define the \textbf{\textcolor{ultramarine}{lower composition of $s_{d+1}M$ and $s_{d+1}\mathbf{H}$ with respect to $s_{d+1}\mathbf{F}$}} as
    the element $(s_{d+1}M\upperset{s_{d+1}\mathbf{F}}{\circ} s_{d+1}\mathbf{H})\in\Multi_d^{\bullet}(\MA,\MB)[d+1]$ \\ \vskip-4mm \noindent such that \[(s_{d+1}M\upperset{s_{d+1}\mathbf{F}}{\circ} s_{d+1}\mathbf{H})^{\doubar{x}}=\sum\mathcal{E}(\mathcalboondox{D})\]where the sum is over all the filled diagrams $\mathcalboondox{D}$ of type $\doubar{x}$ and of the form


    
    Moreover, given $s_{d+1}M'\in \Multi_d^{\bullet}(\MA)[d+1]$ we define the \textbf{\textcolor{ultramarine}{upper composition of $s_{d+1}M'$ and $s_{d+1}\mathbf{H}$ with respect to $s_{d+1}\mathbf{F}$}} as the 
    element $s_{d+1}\mathbf{H}\oversetF{\scriptscriptstyle{s_{d+1}\mathbf{F}}}{\circ} s_{d+1}M'\in\Multi_d^{\bullet}(\MA,\MB)[d+1]$ such that \[(s_{d+1}\mathbf{H}\oversetF{\scriptscriptstyle{s_{d+1}\mathbf{F}}}{\circ} s_{d+1}M')^{\doubar{x}}=\sum\mathcal{E}(\mathcalboondox{D})\]  where the sum is over all the filled diagrams $\mathcalboondox{D}$ of type $\doubar{x}$ and of the form


\end{definition}

\begin{lemma}
\label{lemma:identities-pcy}
Let $s_{d+1}M_{\MA}\in\Multi_d^{\bullet}(\MA)^{C_{\llg(\bullet)}}[d+1]$, $s_{d+1}M_{\MB}\in\Multi_d^{\bullet}(\MB)^{C_{\llg(\bullet)}}[d+1]$ be two elements of degree $1$ and consider $s_{d+1}\mathbf{F}\in\Multi_d^{\bullet}(\MA,\MB)^{C_{\llg(\bullet)}}[d+1]$ of degree $0$.
Then, 
\begin{equation}
    \label{eq:associativity-MA}
    s_{d+1}M_{\MA}\upperset{\scriptscriptstyle{\nec}}{\circ}(s_{d+1}M_{\MA}\upperset{\scriptscriptstyle{\nec}}{\circ}s_{d+1}M_{\MA})=(s_{d+1}M_{\MA}\upperset{\scriptscriptstyle{\nec}}{\circ}s_{d+1}M_{\MA})\upperset{\scriptscriptstyle{\nec}}{\circ}s_{d+1}M_{\MA}
\end{equation}
Moreover, we have that
\begin{equation}
    \label{eq:associativity-F-MB}
    (s_{d+1}M_{\MB}\upperset{\scriptscriptstyle{\pre}}{\circ}s_{d+1}\mathbf{F})\oversetF{\scriptscriptstyle{s_{d+1}\mathbf{F}}}{\circ}s_{d+1}M_{\MA}=s_{d+1}M_{\MB}\uppersetF{\scriptscriptstyle{s_{d+1}\mathbf{F}}}{\circ}(s_{d+1}\mathbf{F}\upperset{\scriptscriptstyle{\multinec}}{\circ}s_{d+1}M_{\MA})
\end{equation}
and 
\begin{equation}
    \label{eq:associativity-F-MA}
     (s_{d+1}\mathbf{F}\upperset{\scriptscriptstyle{\multinec}}{\circ}s_{d+1}M_{\MA})\oversetF{\scriptscriptstyle{s_{d+1}\mathbf{F}}}{\circ}s_{d+1}M_{\MA}=s_{d+1}\mathbf{F}\upperset{\scriptscriptstyle{\multinec}}{\circ}(s_{d+1}M_{\MA}\upperset{\scriptscriptstyle{\nec}}{\circ}s_{d+1}M_{\MA})
\end{equation}
\end{lemma}
\begin{proof}
We first show \eqref{eq:associativity-MA}. 
For every $\doubar{x}\in\doubar{\MO}_{\MA}$ left hand term is given by 
\begin{equation}
\begin{split}
    (s_{d+1}M_{\MA}\upperset{\scriptscriptstyle{\nec}}{\circ}&(s_{d+1}M_{\MA}\upperset{\scriptscriptstyle{\nec}}{\circ}s_{d+1}M_{\MA}))^{\doubar{x}}\\&=\sum\mathcal{E}(\mathcalboondox{D_1})+\sum\mathcal{E}(\mathcalboondox{D_2})+\sum\mathcal{E}(\mathcalboondox{D_3})+\sum\mathcal{E}(\mathcalboondox{D_4})+\sum\mathcal{E}(\mathcalboondox{D_5})-\sum\mathcal{E}(\mathcalboondox{D_6})
\end{split}
\end{equation}
where the sums are over all the filled diagrams $\mathcalboondox{D_i}$ of type $\doubar{x}$ and of the form

\noindent
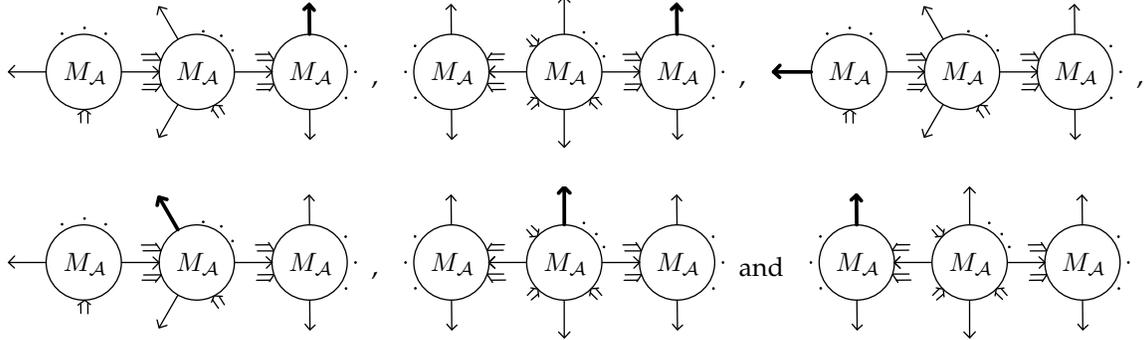
\begin{minipage}{21cm}

\end{minipage}
respectively.
The minus sign before $\sum\mathcal{E}(\mathcalboondox{D_6})$ comes from the fact that the order of the first outgoing arrows of the discs filled with $M_{\MA}$ have changed their order and that $s_{d+1}M_{\MA}$ is of degree $1$. We have that $\sum\mathcal{E}(\mathcalboondox{D_6})=\sum\mathcal{E}(\mathcalboondox{D_2})$.
Moreover, $\sum\mathcal{E}(\mathcalboondox{D_5})=0$. Indeed, the two discs sharing an arrow with the one with the bold arrow are filled with $M_{\MA}$ and are ordered by the labeling of their outgoing arrows. There are two different cases: either the first disc is filled with the first $M_{\MA}$ appearing in the identity \eqref{eq:associativity-MA} or with the second one. Since $s_{d+1}M_{\MA}$ is of degree $1$, the second case gives rise to a minus sign and the two cases then cancel each other.

We thus get 
\begin{equation}
     (s_{d+1}M_{\MA}\upperset{\scriptscriptstyle{\nec}}{\circ}(s_{d+1}M_{\MA}\upperset{\scriptscriptstyle{\nec}}{\circ}s_{d+1}M_{\MA}))^{\doubar{x}}=\sum\mathcal{E}(\mathcalboondox{D_1})+\sum\mathcal{E}(\mathcalboondox{D_3})+\sum\mathcal{E}(\mathcalboondox{D_4})
\end{equation}

On the other hand, the right hand term is given by 
\begin{equation}
\begin{split}
    ((s_{d+1}M_{\MA}\upperset{\scriptscriptstyle{\nec}}{\circ}&s_{d+1}M_{\MA})\upperset{\scriptscriptstyle{\nec}}{\circ}s_{d+1}M_{\MA}))^{\doubar{x}}\\&=\sum\mathcal{E}(\mathcalboondox{D_1'})+\sum\mathcal{E}(\mathcalboondox{D_2'})+\sum\mathcal{E}(\mathcalboondox{D_3'})+\sum\mathcal{E}(\mathcalboondox{D_4'})+\sum\mathcal{E}(\mathcalboondox{D_5'})-\sum\mathcal{E}(\mathcalboondox{D_6'})
\end{split}
\end{equation}
where the sums are over all the filled diagrams $\mathcalboondox{D_i'}$ of type $\doubar{x}$ and of the form

\noindent
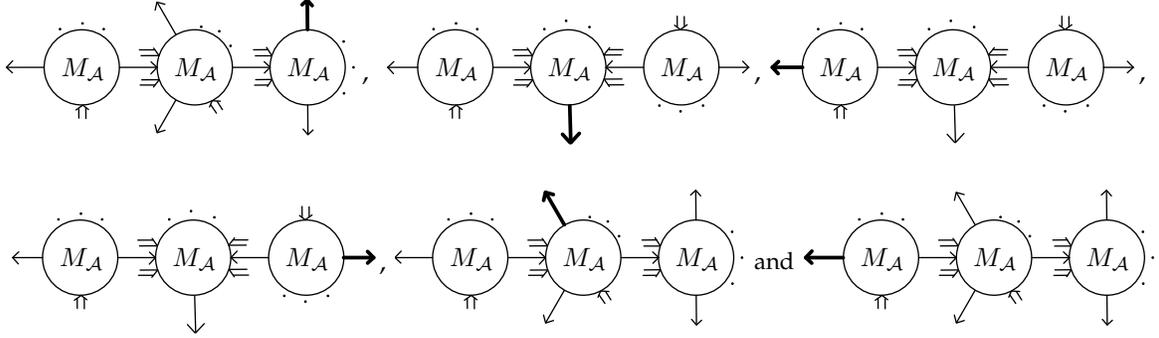
\begin{minipage}{21cm}

\end{minipage}

\noindent Again, the minus before $\sum\mathcal{E}(\mathcalboondox{D_6'})$ comes from the exchange of the order of two $s_{d+1}M_{\MA}$. Moreover, $\sum\mathcal{E}(\mathcalboondox{D_6'})=\sum\mathcal{E}(\mathcalboondox{D_4'})$ and $\sum\mathcal{E}(\mathcalboondox{D_2'})=0$ for the same reason as $\sum\mathcal{E}(\mathcalboondox{D_5})=0$.
We thus get 
 \begin{equation}
     ((s_{d+1}M_{\MA}\upperset{\scriptscriptstyle{\nec}}{\circ}s_{d+1}M_{\MA})\upperset{\scriptscriptstyle{\nec}}{\circ}s_{d+1}M_{\MA}))^{\doubar{x}}=\sum\mathcal{E}(\mathcalboondox{D_1'})+\sum\mathcal{E}(\mathcalboondox{D_3'})+\sum\mathcal{E}(\mathcalboondox{D_5'})
\end{equation}
Since $\sum\mathcal{E}(\mathcalboondox{D_1})=\sum\mathcal{E}(\mathcalboondox{D_1'})$, $\sum\mathcal{E}(\mathcalboondox{D_4})=\sum\mathcal{E}(\mathcalboondox{D_5'})$ and $\sum\mathcal{E}(\mathcalboondox{D_6})=\sum\mathcal{E}(\mathcalboondox{D_3'})$, the identity \eqref{eq:associativity-MA} is proved.

We now prove \eqref{eq:associativity-F-MB}. First, note that 
$((s_{d+1}M_{\MB}\upperset{\scriptscriptstyle{\pre}}{\circ}s_{d+1}\mathbf{F})\oversetF{\scriptscriptstyle{s_{d+1}\mathbf{F}}}{\circ}s_{d+1}M_{\MA})^{\doubar{x}}=\sum\mathcal{E}(\mathcalboondox{D})$
where the \\\vskip-3.8mm\noindent sum is over all the filled diagrams $\mathcalboondox{D}$ of type $\doubar{x}$ and of the form



\noindent It is easy to check that $\big(s_{d+1}M_{\MB}\upperset{\scriptscriptstyle{s_{d+1}F}}{\circ}(s_{d+1}\mathbf{F}\upperset{\scriptscriptstyle{\multinec}}{\circ}s_{d+1}M_{\MA})\big)^{\doubar{x}}=\sum\mathcal{E}(\mathcalboondox{D})$ which proves \\\vskip-3.8mm\noindent \eqref{eq:associativity-F-MB}.
Finally, we have that 
\begin{small}
\begin{equation}
     ((s_{d+1}\mathbf{F}\upperset{\scriptscriptstyle{\multinec}}{\circ}s_{d+1}M_{\MA})\oversetF{\scriptscriptstyle{s_{d+1}\mathbf{F}}}{\circ}s_{d+1}M_{\MA})^{\doubar{x}}=\sum\mathcal{E}(\mathcalboondox{D_1})+\sum\mathcal{E}(\mathcalboondox{D_2})+\sum\mathcal{E}(\mathcalboondox{D_3})+\sum\mathcal{E}(\mathcalboondox{D_4})-\sum\mathcal{E}(\mathcalboondox{D_5})
\end{equation}
\end{small}
where the sums are over all the filled diagrams of type $\doubar{x}$ and of the form

\noindent
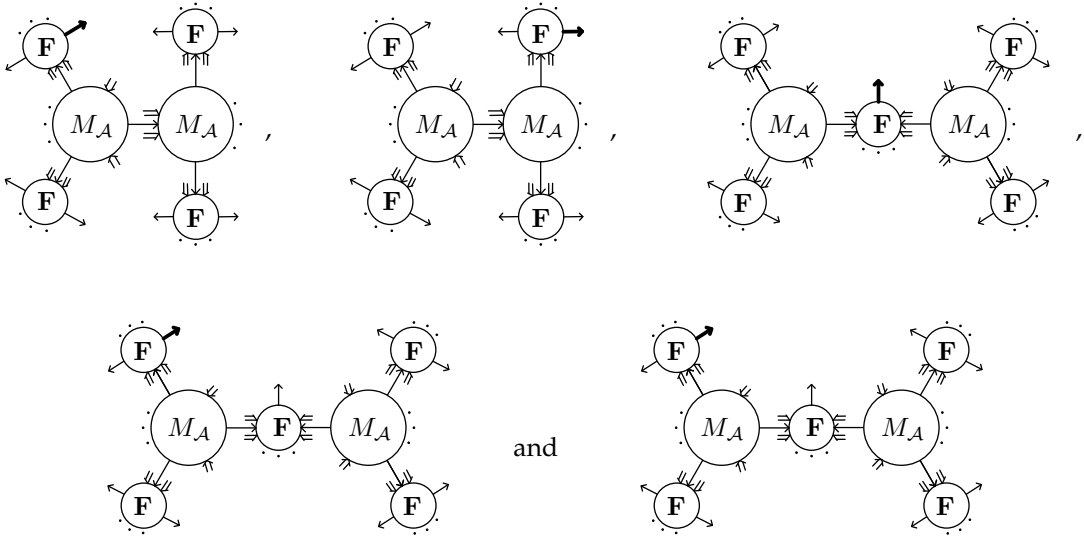
\begin{minipage}{21cm}

\end{minipage}

\noindent
respectively. Note that $\sum\mathcal{E}(\mathcalboondox{D_4})=\sum\mathcal{E}(\mathcalboondox{D_5})$ and $\sum\mathcal{E}(\mathcalboondox{D_3})=0$, so that 
\begin{equation}
     ((s_{d+1}F\upperset{\scriptscriptstyle{\multinec}}{\circ}s_{d+1}M_{\MA})\oversetF{\scriptscriptstyle{s_{d+1}F}}{\circ}s_{d+1}M_{\MA})^{\doubar{x}}=\sum\mathcal{E}(\mathcalboondox{D_1})+\sum\mathcal{E}(\mathcalboondox{D_2})
\end{equation}
It is clear that 
\begin{equation}
     (s_{d+1}F\upperset{\scriptscriptstyle{\multinec}}{\circ}(s_{d+1}M_{\MA}\upperset{\scriptscriptstyle{\nec}}{\circ}s_{d+1}M_{\MA}))^{\doubar{x}}=\sum\mathcal{E}(\mathcalboondox{D_1})+\sum\mathcal{E}(\mathcalboondox{D_2})
\end{equation}
This finishes the proof.
\end{proof}

\begin{corollary}
Let $s_{d+1}M_{\MA}\in\Multi_d^{\bullet}(\MA)^{C_{\llg(\bullet)}}[d+1]$, $s_{d+1}M_{\MB}\in\Multi_d^{\bullet}(\MB)^{C_{\llg(\bullet)}}[d+1]$ be two elements of degree $1$ and consider $s_{d+1}\mathbf{F}\in\Multi_d^{\bullet}(\MA,\MB)^{C_{\llg(\bullet)}}[d+1]$ of degree $0$.
Suppose that $s_{d+1}M_{\MB}$ satisfies the identities $(\operatorname{SI}^{\doubar{x}})$ for each $\doubar{x}\in\doubar{\MO}_{\MB}$.
Then, the following holds.
\begin{small}
\begin{equation}
\begin{split}
    \label{eq:second-identity-pcy}
    s_{d+1}\mathbf{F}\upperset{\scriptscriptstyle{\multinec}}{\circ}(s_{d+1}M_{\MA}\upperset{\scriptscriptstyle{\nec}}{\circ}s_{d+1}M_{\MA})&=(s_{d+1}\mathbf{F}\upperset{\scriptscriptstyle{\multinec}}{\circ}s_{d+1}M_{\MA}-s_{d+1}M_{\MB}\upperset{\scriptscriptstyle{\pre}}{\circ} s_{d+1}\mathbf{F})\oversetF{\scriptscriptstyle{s_{d+1}\mathbf{F}}}{\circ} s_{d+1}M_{\MA}\\&+s_{d+1}M_{\MB}\uppersetF{\scriptscriptstyle{s_{d+1}\mathbf{F}}}{\circ}(s_{d+1}\mathbf{F}\upperset{\scriptscriptstyle{\multinec}}{\circ}s_{d+1}M_{\MA}-s_{d+1}M_{\MB}\upperset{\scriptscriptstyle{\pre}}{\circ} s_{d+1}\mathbf{F})
    \end{split}
\end{equation}
\end{small}
\end{corollary}
\begin{proof}
    It suffices to show that $s_{d+1}M_{\MB}\uppersetF{\scriptscriptstyle{s_{d+1}\mathbf{F}}}{\circ}(s_{d+1}M_{\MB}\upperset{\scriptscriptstyle{\pre}}{\circ} s_{d+1}\mathbf{F})=0$. 
    By definition, for every $\doubar{x}\in\doubar{\MO}_{\MA}$ \\ \vskip-3.6mm\noindent we have $(s_{d+1}M_{\MB}\uppersetF{\scriptscriptstyle{s_{d+1}\mathbf{F}}}{\circ}(s_{d+1}M_{\MB}\upperset{\scriptscriptstyle{\pre}}{\circ} s_{d+1}\mathbf{F}))^{\doubar{x}}=\sum\mathcal{E}(\mathcalboondox{D})$ where the sum is over all the filled diagrams \\ \vskip-3.6mm\noindent $\mathcalboondox{D}$ of type $\doubar{x}$ and of the form



\noindent 
Then, $s_{d+1}M_{\MB}\uppersetF{\scriptscriptstyle{s_{d+1}\mathbf{F}}}{\circ}(s_{d+1}M_{\MB}\upperset{\scriptscriptstyle{\pre}}{\circ} s_{d+1}\mathbf{F})=(s_{d+1}M_{\MB}\upperset{\scriptscriptstyle{G}}{\circ}s_{d+1}M_{\MB})\upperset{\scriptscriptstyle{\pre}}{\circ} s_{d+1}\mathbf{F}$ which vanishes since \\\vskip-3.8mm\noindent $s_{d+1}M_{\MB}$ satisfies the Stasheff identites by assumption.
\end{proof}

We now prove a version of the Homotopy Transfer Theorem for pre-Calabi-Yau categories, adapting the techniques of \cite{petersen} presented in Section \ref{A-inf-case}.

\begin{theorem}
\label{thm:2-pcy}
Let $(\MA,d_{\MA})$, $(\MB,d_{\MB})$ be dg quivers and $(F_0,f):\MA\rightarrow \MB$ be a quasi-isomorphism of dg quivers.
Let $s_{d+1}M_{\MB}\in\Multi_d^{\bullet}(\MB)^{C_{\llg(\bullet)}}[d+1]$ be an element of degree $1$, with $s_{d+1}M_{\MB}^{x,y}=d_{\MB}^{x,y}[1]$ for every $x,y\in\MO_{\MB}$, that turns $\MB$ into a $d$-pre-Calabi-Yau category.
Then, there exists a $d$-pre-Calabi-Yau structure $s_{d+1}M_{\MA}$ on $\MA$ such that $s_{d+1}M_{\MA}^{x,y}=d_{\MA}^{x,y}[1]$ for every $x,y\in\MO_{\MA}$ as well as a $d$-pre-Calabi-Yau morphism $(F_0,s_{d+1}\mathbf{F})$ such that $s_{d+1}F^{x,y}=f^{x,y}[1]$ for every $x,y\in\MO_{\MA}$. 
\end{theorem}
\begin{proof}
    We construct the $d$-pre-Calabi-Yau structure and the $d$-pre-Calabi-Yau morphism inductively.
    We already have the maps for $\bar{x}\in\bar{\MO}_{\MA}$ by Theorem \ref{thm2-Ainf}. Consider $\doubar{x}=(\bar{x}^1,\dots,\bar{x}^n)\in\bar{\MO}^n_{\MA}$ with $n>1$ and suppose that we have constructed $s_{d+1}M_{\MA}\in\Multi_d^{\bullet}(\MA)^{C_{\llg(\bullet)}}[d+1]$ of degree $1$ and $s_{d+1}\mathbf{F}\in\Multi_d^{\bullet}(\MA,\MB)^{C_{\llg(\bullet)}}[d+1]$ of degree $0$ such that the identities $(\operatorname{SI}^{\doubar{y}})$ and $(\operatorname{MI}^{\doubar{y}})$ are satisfied for every $\doubar{y}\in\doubar{\MO}_{\MA}$ such that $\llg(\doubar{y})<\llg(\doubar{x})$ or $\llg(\doubar{y})=\llg(\doubar{x})$ and $N(\doubar{y})<N(\doubar{x})$.
    Since $(s_{d+1}M_{\MA}\upperset{\scriptscriptstyle{\nec}}{\circ} s_{d+1}M_{\MA})^{\doubar{y}}$ vanishes for $\doubar{y}\in\doubar{\MO}$ such that $\llg(\doubar{y})<\llg(\doubar{x})$ or $\llg(\doubar{y})=\llg(\doubar{x})$ and $N(\doubar{y})<N(\doubar{x})$, \\\vskip-3.8mm\noindent the term indexed by $\doubar{x}$ in the identity \eqref{eq:associativity-MA} is simply
    
    \[d_{{}_{\llt(\bar{x}^1)}\MA_{\rrt(\bar{x}^2)}[-d]\otimes\dots\otimes {}_{\llt(\bar{x}^n)}\MA_{\rrt(\bar{x}^1)}[1]}\circ (s_{d+1}M_{\MA}\upperset{\scriptscriptstyle{\nec}}{\circ} s_{d+1}M_{\MA})^{\doubar{x}}=(s_{d+1}M_{\MA}\upperset{\scriptscriptstyle{\nec}}{\circ} s_{d+1}M_{\MA})^{\doubar{x}}\circ d_{\MA[1]^{\otimes \doubar{x}}}\]
   and we thus have that $\partial_{\MA}(s_{d+1}M_{\MA}\upperset{\scriptscriptstyle{\nec}}{\circ} s_{d+1}M_{\MA})^{\doubar{x}}=0$, where we have denoted by $\partial_{\MA}$ the differential \\ \vskip-3.8mm \noindent of $\Multi_d^{\bullet}(\MA)[d+1]$, \textit{i.e.} $(s_{d+1}M_{\MA}\upperset{\scriptscriptstyle{\nec}}{\circ} s_{d+1}M_{\MA})^{\doubar{x}}$ is a cycle. \\ \vskip-3.8mm\noindent

    Moreover, the term indexed by $\doubar{x}$ of the identity \eqref{eq:second-identity-pcy} is 
   \begin{equation}
   \begin{split}
    (f^{\llt(\bar{x}^1),\rrt(\bar{x}^2)}[-d]\otimes\dots\otimes& f^{\llt(\bar{x}^n),\rrt(\bar{x}^1)}[1])\circ(s_{d+1}M_{\MA}\upperset{\scriptscriptstyle{\nec}}{\circ}s_{d+1}M_{\MA})^{\doubar{x}}\\&=(s_{d+1}\mathbf{F}\upperset{\scriptscriptstyle{\multinec}}{\circ}s_{d+1}M_{\MA}-s_{d+1}M_{\MB}\upperset{\scriptscriptstyle{\pre}}{\circ} s_{d+1}\mathbf{F})^{\doubar{x}}\circ d_{A[1]^{\otimes \doubar{x}}}
    \\&\phantom{=}+d_{A[-d]^{\otimes (k-1)}\otimes{A[1]}}\circ(s_{d+1}\mathbf{F}\upperset{\scriptscriptstyle{\multinec}}{\circ}s_{d+1}M_{\MA}-s_{d+1}M_{\MB}\upperset{\scriptscriptstyle{\pre}}{\circ} s_{d+1}\mathbf{F})^{\doubar{x}}
     \end{split}
\end{equation}
since $(s_{d+1}\mathbf{F}\upperset{\scriptscriptstyle{\multinec}}{\circ}s_{d+1}M_{\MA}-s_{d+1}M_{\MB}\upperset{\scriptscriptstyle{\pre}}{\circ} s_{d+1}\mathbf{F})^{\doubar{y}}$ vanishes for $\doubar{y}\in\doubar{\MO}_{\MA}$ such that $\llg(\doubar{y})<\llg(\doubar{x})$ or \\\vskip-3.8mm\noindent $\llg(\doubar{y})=\llg(\doubar{x})$ and $N(\doubar{y})<N(\doubar{x})$.\\\vskip-3.8mm\noindent

Then, $(f^{\llt(\bar{x}^1),\rrt(\bar{x}^2)}[-d]\otimes\dots\otimes f^{\llt(\bar{x}^n),\rrt(\bar{x}^1)}[1])\circ(s_{d+1}M_{\MA}\upperset{\scriptscriptstyle{\nec}}{\circ}s_{d+1}M_{\MA})^{\doubar{x}}$ is a boundary, and since \\\vskip-3.8mm\noindent $(s_{d+1}M_{\MA}\upperset{\scriptscriptstyle{\nec}}{\circ}s_{d+1}M_{\MA})^{\doubar{x}}$ is a cycle and $f$ is a quasi-isomorphism, $(s_{d+1}M_{\MA}\upperset{\scriptscriptstyle{\nec}}{\circ}s_{d+1}M_{\MA})^{\doubar{x}}$ is itself a \\\vskip-3.8mm\noindent boundary, 
meaning that there exists $E^{\doubar{x}}\in \Multi_d^{\doubar{x}}(\MA)$ such that $(s_{d+1}M_{\MA}\upperset{\scriptscriptstyle{\nec}}{\circ}s_{d+1}M_{\MA})^{\doubar{x}}=\partial_{\MA}(s_{d+1}E^{\doubar{x}})$. \\\vskip-3.8mm\noindent
We thus define $\mu_{\MA}^{\doubar{y}}=M_{\MA}^{\doubar{y}}$ for $\doubar{y}\in\doubar{\MO}$ such that $\llg(\doubar{y})<\llg(\doubar{x})$ or $\llg(\doubar{y})=\llg(\doubar{x})$ and $N(\doubar{y})<N(\doubar{x})$ and $\mu_{\MA}^{\doubar{x}}=M_{\MA}^{\doubar{x}}-E$. 
Then, $(s_{d+1}\mu_A\upperset{\scriptscriptstyle{\nec}}{\circ}s_{d+1}\mu_A)^{\doubar{y}}=0$ for $\doubar{y}\in\doubar{\MO}$ such that $\llg(\doubar{y})<\llg(\doubar{x})$ or $\llg(\doubar{y})=\llg(\doubar{x})$ \\\vskip-3.8mm\noindent and $N(\doubar{y})\leq N(\doubar{x})$. It is clear that $\mu_{\MA}$ is invariant under the action of the cyclic group since $M_{\MA}$ is and the element $E=(E^{\doubar{x}})_{\doubar{x}\in\doubar{\MO}_{\MA}}\in\Multi_d^{\bullet}(\MA)$ is invariant since $s_{d+1}M_{\MA}\upperset{\scriptscriptstyle{\nec}}{\circ}s_{d+1}M_{\MA}$ is.

Moreover, by \eqref{eq:second-identity} we have that 
$(s_{d+1}\mathbf{F}\upperset{\scriptscriptstyle{\multinec}}{\circ}s_{d+1}M_{\MA} -s_{d+1}M_{\MB}\upperset{\scriptscriptstyle{\pre}}{\circ}s_{d+1}\mathbf{F})^{\doubar{x}}$ is a cycle so it \\\vskip-4mm\noindent can be written as 
\begin{small}
    \begin{equation}
        (s_{d+1}\mathbf{F}\upperset{\scriptscriptstyle{\multinec}}{\circ}s_{d+1}M_{\MA} -s_{d+1}M_{\MB}\upperset{\scriptscriptstyle{\pre}}{\circ}s_{d+1}\mathbf{F})^{\doubar{x}}=(f^{\llt(\bar{x}^1),\rrt(\bar{x}^2)}[-d]\otimes\dots\otimes f^{\llt(\bar{x}^n),\rrt(\bar{x}^1)}[1])\circ s_{d+1}E'+ \partial (s_{d+1}E\dprime)
    \end{equation}
\end{small} where $E'$ is a cycle of $\Multi_d^{\doubar{x}}(\MA)$, $E\dprime\in \Multi_d^{\doubar{x}}(\MA,\MB)$ and $\partial$ denotes the differential of $\Multi_d^{\bullet}(\MA,\MB)[d+1]$.

We set $\nu_{\MA}^{\doubar{y}}=\mu_{\MA}^{\bar{i}}$, $\mathbf{G}^{\doubar{y}}=\mathbf{F}^{\doubar{y}}$ for $\doubar{y}\in\doubar{\MO}$ such that $\llg(\doubar{y})<\llg(\doubar{x})$ or $\llg(\doubar{y})=\llg(\doubar{x})$ and $N(\doubar{y})<N(\doubar{x})$ and set $\nu_{\MA}^{\doubar{x}}=\nu_{\doubar{x}}^{\doubar{x}}-E'$, $\mathbf{G}^{\doubar{x}}=\mathbf{F}^{\doubar{x}}-E\dprime$.
Therefore, since we have only modified $\mu_A$ by adding a cycle, we have $(s_{d+1}\nu_{\MA}\upperset{\nec}{\circ}s_{d+1}\nu_{\MA})^{\doubar{y}}=0$ for $\doubar{y}\in\doubar{\MO}_{\MA}$ such that $\llg(\doubar{y})<\llg(\doubar{x})$ or $\llg(\doubar{y})=\llg(\doubar{x})$ and
\\\vskip-3.8mm\noindent $N(\doubar{y})\leq N(\doubar{x})$.
It is straightforward to check that $(s_{d+1}\mathbf{G}\upperset{\scriptscriptstyle{\multinec}}{\circ}s_{d+1}\nu_A-s_{d+1}M_{\MB}\upperset{\scriptscriptstyle{\pre}}{\circ}s_{d+1}\mathbf{G})^{\doubar{x}}=0$.
\end{proof}

\begin{theorem}
\label{thm:3-pcy}
Let $(\MA,d_{\MA})$, $(\MB,d_{\MB})$ be dg quivers and $(F_0,f):\MA\rightarrow \MB$ be a quasi-isomorphism of dg quivers.
Let $s_{d+1}M_{\MA}\in\Multi_d^{\bullet}(\MA)^{C_{\llg(\bullet)}}[d+1]$ be an element of degree $1$, with $s_{d+1}M_{\MA}^{x,y}=d_{\MA}^{x,y}[1]$ for every $x,y\in\MO_{\MA}$, that turns $\MA$ into a $d$-pre-Calabi-Yau category.
Then, there exists a $d$-pre-Calabi-Yau structure $s_{d+1}M_{\MB}$ on $\MB$ such that $s_{d+1}M_{\MB}^{x,y}=d_{\MB}^{x,y}[1]$ as well as a $d$-pre-Calabi-Yau morphism $(F_0,s_{d+1}\mathbf{F})$ such that $s_{d+1}F^{x,y}=f^{x,y}[1]$ for every $x,y\in\MO_{\MB}$. 
\end{theorem}
\begin{proof}
    We construct the $d$-pre-Calabi-Yau structure and the $d$-pre-Calabi-Yau morphism inductively.
    We already have the maps for $\bar{x}\in\bar{\MO}_{\MA}$ by Theorem \ref{thm3-Ainf}. Consider $\doubar{x}=(\bar{x}^1,\dots,\bar{x}^n)\in\bar{\MO}^n_{\MA}$ with $n>1$. Given $\bar{y}=(y_1,\dots,y_n)\in\MO^n_{\MB}$ we denote $\xoverline{F_0(\bar{y})}=(F_0(y_1),\dots,F_0(y_n))$ and given $\doubar{y}=(\bar{y}^1,\dots,\bar{y}^n)\in\bar{\MO}^n_{\MA}$, we denote $\doubarl{F_0(\doubar{y})}=(\xoverline{F_0(\bar{y}^1)},\dots,\xoverline{F_0(\bar{y}^n)})$.
    
    Suppose that we have constructed an element $s_{d+1}M_{\MB}\in\Multi_d^{\bullet}(\MB)^{C_{\llg(\bullet)}}[d+1]$ of degree $1$ and $s_{d+1}\mathbf{F}\in\Multi_d^{\bullet}(\MA,\MB)^{C_{\llg(\bullet)}}[d+1]$ such that the identities $(\operatorname{SI}^{\doubarl{\scriptscriptstyle{F_0(\doubar{y})}}})$ and $(\operatorname{MI}^{\doubar{y}})$ are satisfied for every $\doubar{y}\in\doubar{\MO}_{\MA}$ such that $\llg(\doubar{y})<\llg(\doubar{x})$ or $\llg(\doubar{y})=\llg(\doubar{x})$ and $N(\doubar{y})<N(\doubar{x})$.
   
    The term indexed by $\doubarl{F_0(\doubar{x})}$ of the identity \eqref{eq:associativity-MA} applied with $s_{d+1}M_{\MB}$ shows as before that the element $(s_{d+1}M_{\MB}\upperset{\scriptscriptstyle{\nec}}{\circ} s_{d+1}M_{\MB})^{\doubarl{\scriptscriptstyle{F_0(\doubar{x})}}}$ is a cycle.
    
 We now consider the following identity.
 \begin{small}
    \begin{equation}
        \label{eq:thm3-pCY}
        \begin{split}
        (s_{d+1}M_{\MB}\upperset{\scriptscriptstyle{\nec}}{\circ}s_{d+1}M_{\MB})\upperset{\scriptscriptstyle{\pre}}{\circ}s_{d+1}\mathbf{F}=&(s_{d+1}M_{\MB}\upperset{\scriptscriptstyle{\pre}}{\circ}s_{d+1}\mathbf{F}-s_{d+1}\mathbf{F}\upperset{\scriptscriptstyle{\multinec}}{\circ} s_{d+1}M_{\MA})\oversetF{\scriptscriptstyle{s_{d+1}F}}{\circ}s_{d+1}M_{\MA}
        \phantom{=}\\&+s_{d+1}M_{\MB}\uppersetF{\scriptscriptstyle{s_{d+1}\mathbf{F}}}{\circ}(s_{d+1}M_{\MB}\upperset{\scriptscriptstyle{\pre}}{\circ}s_{d+1}\mathbf{F}-s_{d+1}\mathbf{F}\upperset{\scriptscriptstyle{\multinec}}{\circ}s_{d+1}M_{\MA})
          \end{split}
    \end{equation}
    \end{small}
    Its term indexed by $\doubar{x}$ is 
   \begin{equation}
   \begin{split}
    &(s_{d+1}M_{\MB}\upperset{\scriptscriptstyle{\nec}}{\circ}s_{d+1}M_{\MB})^{\doubarl{\scriptscriptstyle{F_0(\doubar{x})}}}\circ (f^{\rrt(\bar{x}^2),\llt(\bar{x}^1)}[1]\otimes\dots\otimes f^{\rrt(\bar{x}^1),\llt(\bar{x}^n)}[1])\\&=(s_{d+1}M_{\MB}\upperset{\scriptscriptstyle{\pre}}{\circ} s_{d+1}\mathbf{F}-s_{d+1}\mathbf{F}\upperset{\scriptscriptstyle{\multinec}}{\circ}s_{d+1}M_{\MA})^{\doubar{x}}\circ d_{A[1]^{\otimes \doubar{x}}}
    \\&\phantom{=}+d_{{}_{\llt(\bar{x}^1)}\MA_{\rrt(\bar{x}^2)}[-d]\otimes\dots\otimes {}_{\llt(\bar{x}^n)}\MA_{\rrt(\bar{x}^1)}[1]}\circ(s_{d+1}M_{\MB}\upperset{\scriptscriptstyle{\pre}}{\circ} s_{d+1}\mathbf{F}-s_{d+1}\mathbf{F}\upperset{\scriptscriptstyle{\multinec}}{\circ}s_{d+1}M_{\MA})^{\doubar{x}}
     \end{split}
\end{equation}
since $(s_{d+1}M_{\MB}\upperset{\scriptscriptstyle{\pre}}{\circ} s_{d+1}\mathbf{F}-s_{d+1}\mathbf{F}\upperset{\scriptscriptstyle{\multinec}}{\circ}s_{d+1}M_{\MA})^{\doubar{y}}$ vanishes for every $\doubar{y}\in\doubar{\MO}_{\MA}$ such that $\llg(\doubar{y})<\llg(\doubar{x})$ \\\vskip-3.8mm\noindent or $\llg(\doubar{y})=\llg(\doubar{x})$ and $N(\doubar{y})<N(\doubar{x})$.

Then, \[(s_{d+1}M_{\MB}\upperset{\scriptscriptstyle{\nec}}{\circ}s_{d+1}M_{\MB})^{\doubarl{\scriptscriptstyle{F_0(\doubar{x})}}}\circ (f^{\rrt(\bar{x}^2),\llt(\bar{x}^1)}[1]\otimes\dots\otimes f^{\rrt(\bar{x}^1),\llt(\bar{x}^n)}[1])
\]
is a boundary so there exists $E\in \Multi_d^{\doubarl{\scriptscriptstyle{F_0(\doubar{x})}}}(B)$ such that $(s_{d+1}M_{\MB}\upperset{\scriptscriptstyle{\nec}}{\circ}s_{d+1}M_{\MB})^{\doubarl{\scriptscriptstyle{F_0(\doubar{x})}}}=\partial (s_{d+1}E)$. \\\vskip-3.8mm\noindent
We thus define $\mu_{\MB}^{\doubarl{\scriptscriptstyle{F_0(\doubar{y})}}}=M_{\MB}^{\doubarl{\scriptscriptstyle{F_0(\doubar{y})}}}$ for every $\doubar{y}\in\doubar{\MO}_{\MA}$ such that $\llg(\doubar{y})<\llg(\doubar{x})$ or $\llg(\doubar{y})=\llg(\doubar{x})$ and \\\vskip-3.8mm\noindent $N(\doubar{y})<N(\doubar{x})$ and $\mu_{\MB}^{\doubarl{\scriptscriptstyle{F_0(\doubar{x})}}}=M_{\MB}^{\doubarl{\scriptscriptstyle{F_0(\doubar{x})}}}-E$. 
Then, $(s_{d+1}\mu_{\MB}\upperset{\scriptscriptstyle{\nec}}{\circ}s_{d+1}\mu_{\MB})^{\doubarl{\scriptscriptstyle{F_0(\doubar{y})}}}=0$ for every $\doubar{y}\in\doubar{\MO}_{\MA}$ such that \\\vskip-3.8mm\noindent $\llg(\doubar{y})<\llg(\doubar{x})$ or $\llg(\doubar{y})=\llg(\doubar{x})$ and $N(\doubar{y})\leq N(\doubar{x})$.

Moreover, by \eqref{eq:thm3-pCY} we have that 
$(s_{d+1}M_{\MB}\upperset{\scriptscriptstyle{\pre}}{\circ}s_{d+1}\mathbf{F}-s_{d+1}\mathbf{F}\upperset{\scriptscriptstyle{\multinec}}{\circ}s_{d+1}M_{\MA})^{\doubar{x}}$ is a cycle so it can \\\vskip-3.8mm\noindent be written as \begin{small}
    \begin{equation}
    (s_{d+1}M_{\MB}\upperset{\scriptscriptstyle{\pre}}{\circ}s_{d+1}\mathbf{F}-s_{d+1}\mathbf{F}\upperset{\scriptscriptstyle{\multinec}}{\circ}s_{d+1}M_{\MA})^{\doubar{x}}= s_{d+1}E'\circ (f^{\rrt(\bar{x}^2),\llt(\bar{x}^1)}[1]\otimes\dots\otimes f^{\rrt(\bar{x}^1),\llt(\bar{x}^n)}[1])+ \partial (s_{d+1}E\dprime)
    \end{equation}
\end{small}where $E'$ is a cycle of $\Multi_d^{\doubarl{\scriptscriptstyle{F_0(\doubar{x})}}}(\MB)$ and $E\dprime\in \Multi_d^{\doubar{x}}(\MA,\MB)$.

We set $\nu_{\MB}^{\doubar{y}}=\mu_{\MB}^{\doubar{y}}$, $G^{\doubar{y}}=F^{\bar{y}}$ for every $\doubar{y}\in\doubar{\MO}_{\MA}$ such that $\llg(\doubar{y})<\llg(\doubar{x})$ or $\llg(\doubar{y})=\llg(\doubar{x})$ and $N(\doubar{y})< N(\doubar{x})$, $\nu_{\MB}^{\doubarl{\scriptscriptstyle{F_0(\doubar{x})}}}=\nu_{\MB}^{\doubarl{\scriptscriptstyle{F_0(\doubar{x})}}}-E'$ and $G^{\doubar{x}}=F^{\doubar{x}}-E\dprime$.
It is straightforward to check that $s_{d+1}\nu_{\MB}$ is
\\\vskip-4mm\noindent invariant and satisfies the identities $(\operatorname{SI}^{\doubarl{\scriptscriptstyle{F_0(\doubar{y})}}})$ and that $s_{d+1}\mathbf{G}$ satisfies the identity $(\operatorname{MI}^{\doubar{y}})$ for every $\doubar{y}\in\doubar{\MO}_{\MA}$ such that $\llg(\doubar{y})<\llg(\doubar{x})$ or $\llg(\doubar{y})=\llg(\doubar{x})$ and $N(\doubar{y})\leq N(\doubar{x})$.
\end{proof}

\subsection{Minimal models and quasi-isomorphisms of pre-Calabi-Yau categories}
We now remark that the previous results leads to the existence of minimal models for pre-Calabi-Yau categories and prove that any quasi-isomorphism of $d$-pre-Calabi-Yau categories admits a quasi-inverse. 
\begin{definition}
    A $d$-pre-Calabi-Yau category $(\MB,s_{d+1}M_{\MB})$ is 
    \textbf{\textcolor{ultramarine}{minimal}} if $M_{\MB}^{x,y} = 0$ for every $x,y\in\MO_{\MB}$. 
    Given a $d$-pre-Calabi-Yau category $(\MA,s_{d+1}M_{\MA})$, a \textbf{\textcolor{ultramarine}{minimal model of $\MA$}} is a minimal $d$-pre-Calabi-Yau category $(\MB,s_{d+1}M_{\MB})$ together with a quasi-isomorphism $(P_0,s_{d+1}\mathbf{P}) : \MA \rightarrow \MB$ of $d$-pre-Calabi-Yau categories. 
\end{definition}

As for the case of $A_{\infty}$-categories we deduce from the previous results that minimal models of pre-Calabi-Yau categories always exist. 
Indeed, given a $d$-pre-Calabi-Yau category $(\MA,s_{d+1}M_{\MA})$,  
let $(P_0,p) : \MA\rightarrow H(\MA)$ be any quasi-isomorphism of dg quivers. 
By Theorem \ref{thm:3-pcy}, there exists a $d$-pre-Calabi-Yau structure on $H(\MA)$ and a morphism of $d$-pre-Calabi-Yau categories $(P_0,s_{d+1}\mathbf{P}) : \MA \rightarrow H(\MA)$ such that $s_{d+1}P^{x,y} = p^{x,y}[1]$ is a quasi-isomorphism for every $x,y\in\MO_{\MA}$.

\begin{lemma}
    \label{lemma:pCY-structure-cohomology}
         Let $(\MA,s_{d+1}M_{\MA})$ be a $d$-pre-Calabi-Yau category. Given quasi-isomorphisms of dg quivers $(Id,p) : \MA\rightarrow H(\MA)$ and $(Id,i) : H(\MA)\rightarrow \MA$, the $d$-pre-Calabi-Yau structures $s_{d+1}M_{i}$ and $s_{d+1}M_{p}$ on $H(\MA)$ given in Theorem \ref{thm:2-pcy} and \ref{thm:3-pcy} are isomorphic. In particular, given quasi-isomorphisms $(Id,p),(Id,p') : \MA\rightarrow H(\MA)$ (resp. $(Id,i),(Id,i') : H(\MA)\rightarrow \MA$), the $d$-pre-Calabi-Yau structures $s_{d+1}M_{p}$ and $s_{d+1}M_{p'}$ (resp. $s_{d+1}M_i$ and $s_{d+1}M_{i'}$) on $H(\MA)$ given in Theorem \ref{thm:3-pcy} (resp. in Theorem \ref{thm:2-pcy}) are isomorphic.
\end{lemma}
\begin{proof}
    Consider quasi-isomorphisms $(Id,p):\MA\rightarrow H(\MA)$ and $(Id,i) : H(\MA)\rightarrow \MA$. 
    By Theorem \ref{thm:2-pcy}, one can define a $d$-pre-Calabi-Yau structure $s_{d+1}M_{p}$ on $H(\MA)$ as well as a $d$-pre-Calabi-Yau morphism $(Id,s_{d+1}\mathbf{P}_{\MA}) : \MA \rightarrow H(\MA)$ such that $s_{d+1}P^{x,y}=p^{x,y}[1]$ for every $x,y\in\MO_{\MA}$.
    Moreover, by Theorem \ref{thm:3-pcy}, one can construct a  $d$-pre-Calabi-Yau structure $s_{d+1}M_{i}$ on $H(\MA)$ and a $d$-pre-Calabi-Yau morphism $(Id,s_{d+1}\mathbf{I}) : H(\MA) \rightarrow \MA$ such that $s_{d+1}I^{x,y}=i^{x,y}[1]$ for every $x,y\in\MO_{\MA}$.
    
    Those two constructions give isomorphic structures. Indeed, we have that the composition $s_{d+1}\mathbf{P}\circ s_{d+1}\mathbf{I} : (H(\MA),s_{d+1}M_{i})\rightarrow (H(\MA),s_{d+1}M_{p})$ is a $d$-pre-Calabi-Yau morphism such that $(s_{d+1}\mathbf{P}\circ s_{d+1}\mathbf{I})^{x,y}$ is an isomorphism since $H(\MA)$ is minimal so $s_{d+1}\mathbf{P}\circ s_{d+1}\mathbf{I}$ is an isomorphism of $d$-pre-Calabi-Yau categories by Lemma \ref{lemma:iso-pCY}.
\end{proof}

\begin{definition}
\label{def:morphism-homology-pcy}
    Let $(\MA,s_{d+1}M_{\MA})$ and $(\MB,s_{d+1}M_{\MB})$ be $d$-pre-Calabi-Yau categories. 
    We consider quasi-isomorphisms of dg quivers $(Id,i_{\MA}) : H(\MA)\rightarrow \MA$ and $(Id,p_{\MB}) : \MB\rightarrow H(\MB)$. Then, by Theorem \ref{thm:2-pcy} $H(\MA)$ is endowed with a $d$-pre-Calabi-Yau structure $s_{d+1}M_{i_{\MA}}$ and $i_{\MA}$ extends to a $d$-pre-Calabi-Yau morphism $(Id,s_{d+1}\mathbf{I}_{\MA}) : H(\MA)\rightarrow \MA$. Moreover, by Theorem \ref{thm:3-pcy} $H(\MB)$ is endowed with a $d$-pre-Calabi-Yau structure $s_{d+1}M_{p_{\MB}}$ and $p_{\MB}$ extends to a $d$-pre-Calabi-Yau morphism $(Id,s_{d+1}\mathbf{P}_{\MB}) : \MB\rightarrow H(\MB)$.
    Given a $d$-pre-Calabi-Yau morphism $(F_0,s_{d+1}\mathbf{F}) : \MA\rightarrow \MB$ we define $(F_0,s_{d+1}H(\mathbf{F})) : H(\MA)\rightarrow H(\MB)$ where $s_{d+1}H(\mathbf{F})=s_{d+1}\mathbf{P}_{\MB} \circ s_{d+1}\mathbf{F}\circ s_{d+1}\mathbf{I}_{\MA}$. 
\end{definition}
Using the previous results, we deduce that any quasi-isomorphism of $d$-pre-Calabi-Yau categories admits a quasi-inverse.
\begin{theorem}
\label{thm:q-iso-pCY}
   Let $(\MA,s_{d+1}M_{\MA})$ and $(\MB,s_{d+1}M_{\MB})$ be $d$-pre-Calabi-Yau categories and consider a $d$-pre-Calabi-Yau morphism $(F_0,s_{d+1}\mathbf{F}) : (\MA,s_{d+1}M_{\MA})\rightarrow (\MB,s_{d+1}M_{\MB})$. Consider quasi-iso\-mor\-phisms $(Id,i_{\MA}) : H(\MA)\rightarrow \MA$ and $(Id,p_{\MB}) : \MB\rightarrow H(\MB)$ and $(F_0,s_{d+1}H(\mathbf{F})) : (H(\MA),M_{i_{\MA}})\rightarrow (H(\MB),M_{p_{\MB}})$ given in Definition \ref{def:morphism-homology-pcy}.
    Then, if $(F_0,s_{d+1}\mathbf{F})$ is a quasi-iso\-mor\-phism, there exists a quasi-iso\-mor\-phism $(G_0,s_{d+1}\mathbf{G}):(\MB,s_{d+1}M_{\MB})\rightarrow (\MA,s_{d+1}M_{\MA})$ of $d$-pre-Calabi-Yau categories as well as quasi-isomorphisms $(Id,i_{\MB}) : H(\MB)\rightarrow \MB$, $(Id,p_{\MA}) : \MA\rightarrow H(\MA)$ such that $p_{\MA}\circ i_{\MA}=\id_{H(\MA)}$, $p_{\MB}\circ i_{\MB}=\id_{H(\MB)}$ and whose associated $(G_0,s_{d+1}H(\mathbf{G}))$ in the sense of Definition \ref{def:morphism-homology-pcy} is the inverse of $(F_0,s_{d+1}H(\mathbf{F}))$.
\end{theorem}
\begin{proof}
   Since $(F_0,\mathbf{F})$ is a quasi-isomorphism, it induces by Lemma \ref{lemma:iso-pCY} an isomorphism of $d$-pre-Calabi-Yau categories $(F_0,s_{d+1}H(\mathbf{F})) : (H(\MA),M_{i_{\MA}})\rightarrow (H(\MB),M_{p_{\MB}})$ which has an inverse denoted here by $(K_0,s_{d+1}\mathbf{K}) : (H(\MB),M_{p_{\MB}})\rightarrow (H(\MA),M_{i_{\MA}})$. 
   Consider $(Id,i_{\MB}) : H(\MB)\rightarrow \MB$, $(Id,p_{\MA}) : \MA\rightarrow H(\MA)$ such that $p_{\MA}\circ i_{\MA}=\id_{H(\MA)}$ and $p_{\MB}\circ i_{\MB}=\id_{H(\MB)}$. By Theorem \ref{thm:2-pcy} and Theorem \ref{thm:3-pcy}, those maps extend to $d$-pre-Calabi-Yau morphisms $(Id,s_{d+1}\mathbf{I}_{\MA}) : H(\MA)\rightarrow \MA$, $(Id,s_{d+1}\mathbf{P}_{\MA}) : \MA\rightarrow H(\MA)$, $(Id,s_{d+1}\mathbf{I}_{\MB}) : H(\MB)\rightarrow \MB$ and $(Id,s_{d+1}\mathbf{P}_{\MB}) : \MB\rightarrow H(\MB)$ such that  $(s_{d+1}\mathbf{P}_{\MA}\circ s_{d+1}\mathbf{I}_{\MA})^{x,y}=\id_{{}_{x}H(\MA)_{y}[1]}$ for each $x,y\in\MO_{\MA}$ and $(s_{d+1}\mathbf{P}_{\MB}\circ s_{d+1}\mathbf{I}_{\MB})^{x,y}=\id_{{}_{x}H(\MB)_{y}[1]}$ for each $x,y\in\MO_{\MB}$. Moreover, up to an automorphism of the $d$-pre-Calabi-Yau category $H(\MA)$ (resp. $H(\MB)$), it is possible to assume that $s_{d+1}\mathbf{P}_{\MA}\circ s_{d+1}\mathbf{I}_{\MA}=\id$ (resp. $s_{d+1}\mathbf{P}_{\MB}\circ s_{d+1}\mathbf{I}_{\MB}=\id$).
   Then, $(K_0,s_{d+1}\mathbf{G})$ with $s_{d+1}\mathbf{G}=s_{d+1}\mathbf{I}_{\MA}\circ s_{d+1}\mathbf{K}\circ s_{d+1}\mathbf{P}_{\MB}$ gives a quasi-inverse for $(F_0,s_{d+1}\mathbf{F})$.
\end{proof}

\bibliographystyle{alpha}
\bibliography{mybiblio}

\begin{thebibliography}{KTV21}

\bibitem[Bou23]{moi}
Marion Boucrot.
\newblock Morphisms of pre-{C}alabi-{Y}au categories and morphisms of cyclic {$A_{\infty}$}-categories.
\newblock arXiv:2304.13661 [math.KT], 2023.

\bibitem[HLV21]{hlv}
Eric Hoffbeck, Johan Leray, and Bruno Vallette.
\newblock Properadic homotopical calculus.
\newblock {\em Int. Math. Res. Not. IMRN}, (5):3866--3926, 2021.

\bibitem[Kad80]{kadeishvili}
T.~V. Kadei\v{s}vili.
\newblock On the theory of homology of fiber spaces.
\newblock {\em Uspekhi Mat. Nauk}, 35(3(213)):183--188, 1980.
\newblock International Topology Conference (Moscow State Univ., Moscow, 1979).

\bibitem[Kon13]{kontsevich}
Maxim Kontsevich.
\newblock {W}eak {C}alabi-{Y}au algebras.
\newblock Notes taken from the talk at \emph{Conference on Homological Mirror Symmetry}, 2013.

\bibitem[KTV21]{ktv}
Maxim Kontsevich, Alex Takeda, and Yiannis Vlassopoulos.
\newblock {P}re-{C}alabi-{Y}au algebras and topological quantum field theories.
\newblock arXiv:2112.14667 [math.AG], 2021.

\bibitem[LH03]{lefevre}
Kenji Lefèvre-Hasegawa.
\newblock {S}ur les {A}-infinies catégories.
\newblock 2003.

\bibitem[LV22]{lv}
Johan Leray and Bruno Vallette.
\newblock {P}re-{C}alabi--{Y}au algebras and homotopy double poisson gebras.
\newblock arXiv:2203.05062 [math.QA], 2022.

\bibitem[Pet20]{petersen}
Dan Petersen.
\newblock A closer look at {K}adeishvili's theorem.
\newblock {\em High. Struct.}, 4(2):211--221, 2020.

\bibitem[Sei17]{seidel}
Paul Seidel.
\newblock Fukaya {$A_\infty$}-structures associated to {L}efschetz fibrations. {II}.
\newblock In {\em Algebra, geometry, and physics in the 21st century}, volume 324 of {\em Progr. Math.}, pages 295--364. Birkh\"{a}user/Springer, Cham, 2017.

\bibitem[Sta63]{stasheff}
James~Dillon Stasheff.
\newblock Homotopy associativity of {$H$}-spaces. {I}, {II}.
\newblock {\em Trans. Amer. Math. Soc. 108 (1963), 275-292; ibid.}, 108:293--312, 1963.

\bibitem[Sug61]{sugawara}
Masahiro Sugawara.
\newblock On the homotopy-commutativity of groups and loop spaces.
\newblock {\em Mem. Coll. Sci. Univ. Kyoto Ser. A. Math.}, 33:257--269, 1960/61.

\bibitem[TZ07]{tz}
Thomas Tradler and Mahmoud Zeinalian.
\newblock Algebraic string operations.
\newblock {\em $K$-Theory}, 38(1):59--82, 2007.

\end{thebibliography}

\vspace{1cm}

MARION BOUCROT: Univ. Grenoble Alpes, CNRS, IF, 38000 Grenoble, France

\textit{E-mail adress :} marion.boucrot@univ-grenoble-alpes.fr

\end{document}